\documentclass[12pt]{amsart}
\usepackage{amsmath,mathtools,amsthm,amsfonts,bigints,amssymb, mathrsfs, tikz-cd, xcolor, float}
\usepackage{tikz}
\usepackage{subcaption}
\usetikzlibrary{patterns}
\usepackage{pgf,tikz}

\usepackage{pgfplots}
\usepgfplotslibrary{fillbetween}

\usetikzlibrary{arrows}

\newtheorem{theorem}{Theorem}[section]
\newtheorem{lemma}[theorem]{Lemma}
\newtheorem{proposition}[theorem]{Proposition}
\newtheorem{corollary}[theorem]{Corollary}

\theoremstyle{definition}
\newtheorem{definition}[theorem]{Definition}

\newtheorem{remark}[theorem]{Remark}

\newcommand{\what}{\widehat}

\newcommand{\wtilde}{\widetilde}
\newcommand{\R}{\mathbb R}%
\newcommand{\M}{\mathcal M}%
\newcommand{\C}{\mathbb C}%
\newcommand{\D}{\mathfrak{D}}%
\newcommand{\Q}{\mbox{$\mathbb Q$}}%
\newcommand{\N}{\mathbb N}%
\newcommand{\Ss}{\mathcal S}%
\newcommand{\fd}{\mathfrak d}%
\newcommand{\X}{\mathbb X}%
\newcommand{\cH}{\mathcal H}%
\newcommand{\Hb}{\mathbb H}%

\newcommand{\I}{\mathbb I}%
\numberwithin{equation}{section}

\usepackage[colorlinks=true,linkcolor=purple
,citecolor=blue]{hyperref}
\makeatletter

\renewcommand\subsubsection{\@secnumfont}{\bfseries}%
\renewcommand\subsubsection{\@startsection{subsubsection}{3}
  \z@{.5\linespacing\@plus.7\linespacing}{-.5em}%
  {\normalfont\bfseries}}

  \makeatother
\makeatletter
\@namedef{subjclassname@2020}{%
  \textup{2020} Mathematics Subject Classification}
\makeatother

\begin{document}

\title[Transference principles for divergence sets]{Transference principles for divergence sets of dispersive equations and their applications}

\author[Utsav Dewan]{Utsav Dewan}
\address{Department of Mathematics, Indian Institute of Technology Bombay, Powai, Mumbai-400076, India}
\email{utsav@math.iitb.ac.in}



\begin{abstract}
In this article, we study the fractal refinement of the Carleson's problem:
\begin{enumerate}
\item We introduce an equivalence relation on linear dispersive PDEs on $\mathbb{R}^{n}$ based on the high-frequency oscillation of their Fourier multipliers, proving that equations in the same class share the same Hausdorff dimension of divergence sets. As applications, we estimate these dimensions for abstract dispersive families. Consequently, we recover known $1$-dimensional results for linearized Boussinesq and Beam propagations, while establishing new results for the both of them in higher dimensions and for Klein-Gordon propagations in all dimensions. 
\item Next, we study the effect of radial symmetry by obtaining dimension estimates for fractional Schr\"odinger equations with concave phase and radial initial data. Applying a transference principle, dimension estimates are extended to families of abstract dispersive equations. As special cases, we recover the known result for the Beam propagation and also obtain new results for the linearized Boussinesq and the Klein-Gordon propagations in the radial setting.
\item We study analogues of (2) for the $3$-dimensional real hyperbolic space by obtaining dimension estimates for a wide class of dispersive propagations: the Schr\"odinger, fractional Schr\"odinger (with convex and concave phases), Beam, linearized Boussinesq, Wave and Klein-Gordon equations. This is the first result of its kind, outside the setting of flat Euclidean geometry. Our arguments are based on an inter-spatial transference principle of geometric-analytic flavour. In the process, we also obtain quantitative results on the singular sets of Sobolev functions, in the spirit of \v{Z}ubrini\'c, which may be of independent interest.  
\end{enumerate}
\end{abstract}

\maketitle
\tableofcontents

\section{Introduction}
\label{sec1}
One of the most celebrated problems in Harmonic analysis is the Carleson's problem: determining the optimal regularity of the initial condition $f$ of the
Schr\"odinger equation given by
\begin{equation} \label{schrodinger}
\begin{cases}
	 i\frac{\partial u}{\partial t} -\Delta u=0\:,\:\:\:  &(x,t) \in \R^n \times \R\:, \\
	u(\cdot, 0)=f\:, &\text{ on } \R^n \:,
	\end{cases}
\end{equation}
in terms of the index $s$ such that for all $f$ belonging to the inhomogeneous Sobolev space $H^s(\R^n)$, the solution $u$ of (\ref{schrodinger}) converges pointwise to $f$ a.e., that is, 
\begin{equation} \label{pointwise_convergence}
\displaystyle\lim_{t \to 0} u(x,t)= \displaystyle\lim_{t \to 0} e^{-it\Delta}f(x)=f(x)\:,\:\:\text{ almost everywhere }.
\end{equation}
In 1980, Carleson studied this problem for $n=1$ in \cite{C} and obtained the sufficient condition $s \ge 1/4$ for the pointwise convergence (\ref{pointwise_convergence}) to hold. In 1982, Dahlberg-Kenig \cite{DK} showed that $s \ge 1/4$ is also necessary. This completely solved the problem in dimension one and subsequently, posed the question in higher dimensions. After continuous improvements made by several mathematicians \cite{Cowling,Sjolin,Vega,Bourgain1, MVV, Tao, Lee}, the sufficiency of the threshold $s > n/2(n+1)$ for $n \ge 2$ was recently settled by Du et. al. \cite{DGL,DZ}. This bound is sharp except at the endpoint $s =n/2(n+1)$, due to a counterexample by Bourgain \cite{Bourgain}.

A natural refinement of the Carleson's problem was introduced by Sj\"ogren-Sj\"olin \cite{Sjogren} and Barcel\'o-Bennett-Carbery-Rogers \cite{BBCR}: measuring the Hausdorff dimension of the divergence sets in terms of the regularity of the initial data. The divergence set $\D(f)$ for each $f \in H^s(\R^n)$ is the set on which the pointwise convergence fails,
\begin{equation*}
\D(f):=\left\{x \in \R^n : \displaystyle\lim_{t \to 0} e^{-it\Delta}f(x) \ne f(x)\right\}\:,
\end{equation*}
and then one defines the corresponding dimension as,
\begin{equation*}
\fd_n(s):= \displaystyle\sup_{f \in H^s(\R^n)} dim_{\cH}\D(f)\:,
\end{equation*}
where $dim_{\cH}\D(f)$ is the Hausdorff dimension of $\D(f)$.

In \cite{BBCR}, Barcel\'o et. al. obtained the sharp bound for the Schr\"odinger propagation in dimension one:
\begin{equation} \label{schrodinger_dim1}
\fd_1(s) = \begin{cases}
1\:,\:\:& s<\frac{1}{4} \\
1-2s\:,\:\:& \frac{1}{4} \le s < \frac{1}{2} \\
0\:,\:\:& \frac{1}{2} \le s\:. 
\end{cases}
\end{equation}
We note that at the critical index $s=1/4$ for the Carleson's problem, $\fd_1$ jumps from $1$ down to $1/2$.

Alongside the Schrödinger equation, mathematical physics encompasses numerous other significant dispersive PDEs. One of the most prominent examples is given by the Boussinesq equation:
\begin{equation} \label{Boussinesq}
 \frac{\partial^2 u}{\partial t^2} -\frac{\partial^2 u}{\partial x^2} +\frac{\partial^4 u}{\partial x^4}= \frac{\partial^2}{\partial x^2}(u^2)\:,\:\:\:  (x,t) \in \R^2\:.
\end{equation}
The non-linear PDE given by (\ref{Boussinesq}) models  the propagation of long waves on the surface of water with small amplitude \cite{Boussi}. The corresponding linear part is also of profound importance as it serves as a foundational framework for modeling shallow-water wave kinematics, capturing essential frequency dispersion without nonlinear distortions. 

In this regard, recently Li-Li considered the corresponding problem of divergence sets for the linearized Boussinesq equation in \cite{LiLi} and obtained the same sharp size estimates (\ref{schrodinger_dim1}) by essentially replicating the similar oscillatory integral estimates carried out for the Schr\"odinger equation in \cite{BBCR}. The graph of $\fd_1$ for Schr\"odinger and the linearized Boussinesq is presented in red in Figure \ref{figure1}. 

\begin{figure}[h]
\begin{center}
\begin{tikzpicture}
	\draw [->, thick] (0,0) -- (3,0) node [below left] {$s$};
	\draw [->, thick] (0,0) -- (0,3)  node [below left] {$\fd_1(s)$};
	\draw[red] [-, thick] (0,2) -- (0.5,2);
	\draw[red][thick] plot[domain=0.5:1, smooth] (\x,{2-2*\x});
	\draw[red] [-, thick] (1,0) -- (3,0);
	\draw [fill] (0.5,0) circle [radius=.05] node [below] {\scriptsize$\frac{1}{4}$};
	\draw [fill] (0,2) circle [radius=.05] node [left] {\scriptsize$1$};
	\draw [fill] (0,1) circle [radius=.05] node [left] {\scriptsize$\frac{1}{2}$};
	\draw [fill] (1,0) circle [radius=.05] node [below] {\scriptsize$\frac{1}{2}$};
	\draw[gray] [-, dashed] (0.5,0) -- (0.5,2);
	\draw[gray] [-, dashed] (0,1) -- (0.5,1);
\end{tikzpicture}
\end{center}
\caption{$\fd_1$ for Schr\"odinger and linearized Boussinesq}
\label{figure1}
\end{figure}
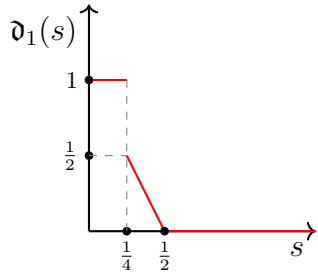

Unlike the one-dimensional situation, in higher dimensions $n \ge 2$, however,  a complete description of $\fd_n(s)$ for the Schr\"odinger propagation is currently not known and one only has certain lower and upper bounds:
\begin{itemize}
\item for $s>n/2$, by the Sobolev embedding, one trivially has $\fd_n(s)=0$;
\item below the critical range $s < n/2(n+1)$, we have $\fd_n(s)=n$, by Bourgain's counter-example \cite{Bourgain}, based on number theorteic arguments;
\item in the range 
 $n/4 \:\le s \: \le n/2$, Barcel\'o
 et. al. obtained the upper bound $\fd_n(s) \le n-2s$ via the Kolmogorov-Seliverstov-Plessner method \cite{BBCR}, which is sharp due to  \v{Z}ubrini\'c's result on the singular sets of Sobolev functions \cite{Zubrinic};
\item in the remaining range there were earlier results on upper bounds by,
\begin{itemize}
\item Barcel\'o
 et. al. using bilinear Fourier extension estimates \cite{BBCR};
\item Luc\`a-Rogers estimating the decay of the spherical averages of the Fourier transform of fractal measures using multilinear estimates \cite{LR2}; 
\item Du-Guth-Li-Zhang by multilinear
refined Strichartz estimates using decoupling and multilinear Kakeya \cite{DGLZ};
\end{itemize}
\item the current best known upper bound:
\begin{equation*}
\fd_n(s) \le n+1 - \frac{2(n+1)s}{n}\:, \:\:\: \frac{n}{2(n+1)} < s < \frac{n}{4}\:,
\end{equation*}
is due to Du et. al. \cite{DGL,DZ} using polynomial partitioning, decoupling and induction on scales;
\item for the corresponding lower bounds, there were earlier results by,
\begin{itemize}
\item Luc\`a-Rogers using ergodic arguments \cite{LR3,LR1};
\item Luc\`a-Ponce-Vanegas adapting Bourgain's number theoretic counter-example to the fractal setting \cite{LP};
\end{itemize} 
\item the current best known lower bounds are due to Eceizabarrena-Ponce-Vanegas \cite{EP} by combining the fractal extension of Bourgain's counter-example as presented in \cite{LP} and the intermediate space trick of Du-Kim-Wang-Zhang \cite{DKWZ}. 
\end{itemize}

The delicate arguments employed above crucially use the fact that the phase of the Fourier multiplier of the Schr\"odinger operator is given by $\psi(\xi):=\|\xi\|^2$ and do not seem to have any straightforward extension to non-polynomial phases, such as, 
\begin{equation} \label{Boussinesq_phase}
\psi(\xi):=\|\xi\| \sqrt{1+\|\xi\|^2}\:,\:\:\:\: \xi \in \R^n\:,
\end{equation}
corresponding to the higher dimensional linearized Boussinesq equation:
\begin{equation*} 
 \frac{\partial^2 u}{\partial t^2} - \Delta u +\Delta^2 u= 0\:,\:\:\:  (x,t) \in \R^n \times \R\:,
\end{equation*}
and the problem of dimension estimates, remains wide open till date.

Now, although the phase given in (\ref{Boussinesq_phase}) is not quadratic, not even polynomial; it is, however, `asymptotically quadratic' in high frequency:
\begin{equation*}
\psi(\xi)=\|\xi\| \sqrt{1+\|\xi\|^2} = \|\xi\|^2 \:+\:\mathcal{O}(1)\:,\:\:\:\: \|\xi\| \gg 1\:,
\end{equation*}
as presented in Figure \ref{figure2}.

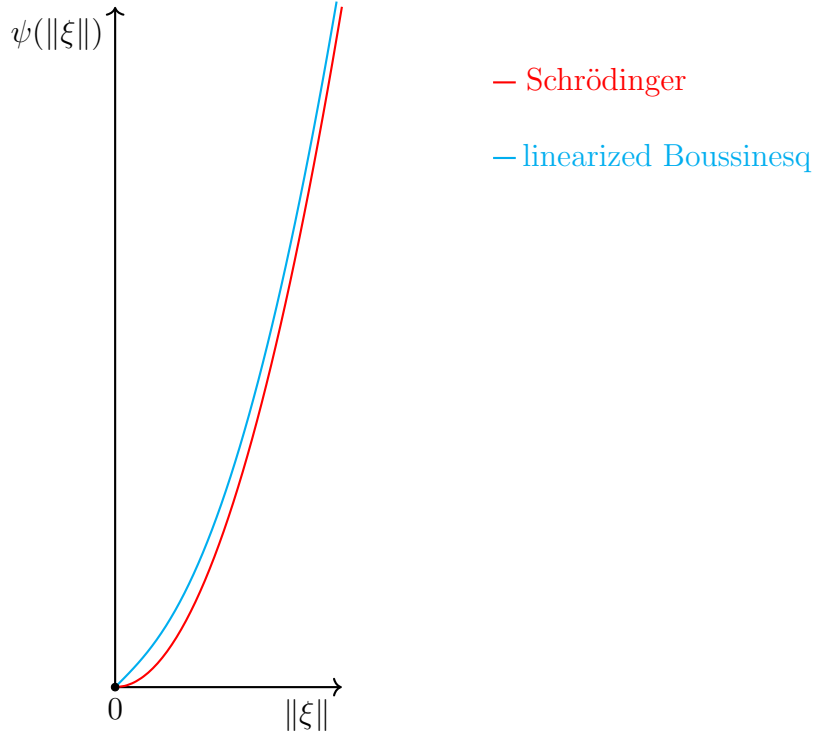
\begin{figure}[h]
\begin{center}
\begin{tikzpicture}
	\draw [->, thick] (0,0) -- (3,0) node [below left] {$\|\xi\|$};
	\draw [->, thick] (0,0) -- (0,9)  node [below left] {$\psi(\|\xi\|)$};
	\draw[red][thick] plot[domain=0:3, smooth] (\x,{(\x)^2});
	\draw[cyan][thick] plot[domain=0:2.93, smooth] (\x,{(\x) * sqrt(1 + (\x)^2)});
	\draw[red] [-, thick] (5,8) -- (5.3,8);
	\pgftext[at={\pgfpoint{6.5cm}{8cm}}]{\textcolor{red}{Schr\"odinger}};
	\draw[cyan] [-, thick] (5,7) -- (5.3,7);
	\pgftext[at={\pgfpoint{7.3cm}{7cm}}]{\textcolor{cyan}{linearized Boussinesq}};
	\draw [fill] (0,0) circle [radius=.05] node [below] {$0$};
\end{tikzpicture}
\end{center}
\caption{Phases of Schr\"odinger and linearized Boussinesq}
\label{figure2}
\end{figure}

This motivates us to ask:

{\em ($\mathcal{Q}$)\:\:What can be concluded about the size estimates of divergence sets of abstract dispersive equations whose phases are not explicit and only their high-frequency asymptotics are known?}

Apart from addressing the genuine non-polynomial scenarios in $\R^n$, the above approach is also very natural when working on negatively curved Riemannian manifolds, as we will later see in the proof of Theorem \ref{H3_theorem} on the 3-dimensional real hyperbolic space. Due to the geometry of the manifold, there is a natural spectral shift of the Laplace-Beltrami operator, causing even the Wave propagation to have a non-polynomial phase and to be subsumed as a non-degenerate Klein-Gordon propagation (see \cite{Anker}).

The above view-point leads us to consider linear dispersive equations:
\begin{equation} \label{dispersive}
\begin{cases}
	 i\frac{\partial u}{\partial t} +\psi(\sqrt{-\Delta} )u=0\:,\:  &(x,t) \in \R^n \times \R\:, \\
	u(\cdot,0)=f\:,\: &\text{ on } \R^n \:,
	\end{cases}
\end{equation}
where $\psi: (0,\infty) \to \R$ is measurable and define:
\begin{definition} \label{defn}
Two dispersive equations of the form (\ref{dispersive}) corresponding to $\psi_1,\psi_2$ are defined to be {\it of comparable oscillation} if there exist $\Lambda>0$ and $C>0$, such that 
\begin{equation*}
\left|\psi_1(\|\xi\|)-\psi_2(\|\xi\|)\right| \le C\:, \text{ for all } \|\xi\| > \Lambda\:.
\end{equation*} 
\end{definition}
Now setting the notation for the $\psi$-dispersive propagations,
\begin{equation*}
\fd_{\psi,n}(s):= \displaystyle\sup_{f \in H^s(\R^n)} dim_{\cH}\D_{\psi}(f)\:,
\end{equation*}
where 
\begin{equation*}
\D_{\psi}(f):=\left\{x \in \R^n : \displaystyle\lim_{t \to 0} e^{it\psi(\sqrt{-\Delta})}f(x) \ne f(x)\right\}\:,\:\: f \in H^s(\R^n)\:,
\end{equation*}
we present our first result which says that dispersive equations belonging to the same equivalence class under `comparable oscillation', have the same $\fd_{\psi,n}$:

\begin{theorem} \label{transference_theorem}
If two dispersive equations with phases $\psi_1,\psi_2$ are of comparable oscillation, then
\begin{equation*}
\fd_{\psi_1,n}(s)=\fd_{\psi_2,n}(s),\:\:\:\text{ for all } s>0\:.
\end{equation*}
\end{theorem}

\begin{remark}
As applications of Theorem \ref{transference_theorem},
\begin{enumerate}
\item  we obtain Hausdorff dimension estimates of divergence sets for families of abstract dispersive equations, classified solely in terms of the oscillation of their Fourier multipliers in high frequency (Corollaries \ref{corollary1}, \ref{corollary2}, \ref{corollary3} and \ref{corollary4});
\item as special cases,
\begin{enumerate}
\item we recover the known one-dimensional results for the linearized Boussinesq and the Beam propagations (Remark \ref{recovering_results_beam});
\item we also obtain genuinely new results for the linearized Boussinesq and the Beam propagations in higher dimensions (Remark \ref{new_results_beam}) and the Klein-Gordon propagations in all dimensions (Remark \ref{new_results_klein}).
\end{enumerate}
\end{enumerate}
For more details, we refer to Section \ref{sec5}.
\end{remark}

Now unlike in the case of the Schr\"odinger equation, not much is known about the fractional Schr\"odinger equations in higher dimensions. However, it is a common phenomenon in harmonic analysis that geometric symmetry enables easier access to analysis of PDEs. This viewpoint was explored by Bennett-Rogers in \cite{BR} where they considered radial initial data on $\R^n$ for the fractional Schr\"odinger equations (\ref{frac_schrodinger}) with convex phase i.e. $a>1$ and obtained the sharp size estimates:
\begin{equation} \label{BR_estimates}
\wtilde{\fd}_{\psi,n}(s)= \begin{cases}
n,\:\: &s<\frac{1}{4}\\
n-2s,\:\: &\frac{1}{4} \le s<\frac{1}{2}\\
0,\:\: &\frac{1}{2} < s\:,
\end{cases}
\end{equation} 
where 
\begin{equation} \label{radial_dimension_defn}
\wtilde{\fd}_{\psi,n}(s):= \displaystyle\sup_{\substack{f \in H^s(\R^n) \\ f \text{   radial }}} dim_{\cH}\D_{\psi}(f)\:.
\end{equation}
\begin{remark}
We note that, in (\ref{radial_dimension_defn}) the supremum is taken only over the class of radial elements in $H^s(\R^n)$ and thus trivially
\begin{equation*}
\wtilde{\fd}_{\psi,n} \le \fd_{\psi,n}\:.
\end{equation*}
But nothing more can be concluded from $\fd_{\psi,n}$ and hence, estimating $\wtilde{\fd}_{\psi,n}$ is a variation of the original problem, not a mere special case.
\end{remark}
Our next result complements the result of Bennett-Rogers for the other fractional powers:
\begin{theorem} \label{radial_concave_thm}
For the fractional Schr\"odinger equations (\ref{frac_schrodinger}) with radial initial data, we have
\begin{enumerate}
\item for concave phase i.e. $a \in (0,1)$,
\begin{enumerate}
\item the sharp result in the range,
\begin{equation*}
\wtilde{\fd}_{\psi,n}(s) =\begin{cases}
n\:,\:\: &s<\frac{a}{4}\\
n-2s\:,\:\: &\frac{1}{4} \le s<\frac{1}{2}\\
0\:,\:\: &\frac{1}{2} < s\:;
\end{cases}
\end{equation*}
\item and in the remaining range, \\
the upper bound,
\begin{equation*}
\wtilde{\fd}_{\psi,n}(s) \le n-\frac{1}{2}+\frac{1-4s}{2(1-a)}\:,\:\: \frac{a}{4} \le s \le \frac{1}{4}\:,
\end{equation*}
and the lower bound,
\begin{equation*}
\wtilde{\fd}_{\psi,n}(s) \ge n-2s\:,\:\: \frac{a}{4} \le s \le \frac{1}{4}\:;
\end{equation*}
\end{enumerate}
\item for the Wave propagation i.e. $a=1$,
\begin{equation*}
\wtilde{\fd}_{\psi,n}(s) =\begin{cases}
n\:,\:\: &s\le \frac{1}{2}\\
0\:,\:\: &\frac{1}{2} < s\:.
\end{cases}
\end{equation*}
\end{enumerate}
\end{theorem}
The results for $a>1$ given by (\ref{BR_estimates}) and for $a \in (0,1),\:a=1$ given by Theorem \ref{radial_concave_thm} are presented in Figure \ref{radialfigure1}.

\begin{figure}[htbp]
  \centering
  
  \begin{subfigure}{0.3\textwidth}
    \centering
    \begin{tikzpicture}
\draw [->, thick] (0,2) -- (2,2) node [below left] {$s$};
\draw [->, thick] (0,2) -- (0,7)  node [below left] {$\wtilde{\fd}_{\psi,n}(s)$};
\draw[red] [-, thick] (0,6) -- (0.5,6);
\draw[red][thick] plot[domain=0.5:1, smooth] (\x,{6-2*\x});
\draw[red] [-, thick] (1,2) -- (2,2);
\draw [fill] (1,2) circle [radius=.05] node [below] {\scriptsize$\frac{1}{2}$};
\draw [fill] (0.5,2) circle [radius=.05] node [below] {\scriptsize$\frac{1}{4}$};
\draw [fill] (0,6) circle [radius=.05] node [left] {\scriptsize$n$};
\draw [fill] (0,5) circle [radius=.05] node [left] {\scriptsize$n-\frac{1}{2}$};
\draw [fill] (0,4) circle [radius=.05] node [left] {\scriptsize$n-1$};
\draw[gray] [-, dashed] (1,2) -- (1,4);
\draw[gray] [-, dashed] (0.5,2) -- (0.5,6);
\draw[gray] [-, dashed] (0,4) -- (1,4);
\draw[gray] [-, dashed] (0,5) -- (0.5,5);
\end{tikzpicture}
    \caption{$a>1$}
    \label{radialfigure1.1}
  \end{subfigure}
  \hfill 
  \begin{subfigure}{0.3\textwidth}
    \centering
    \begin{tikzpicture}
\draw [->, thick] (0,2) -- (2,2) node [below left] {$s$};
\draw [->, thick] (0,2) -- (0,7)  node [below left] {$\wtilde{\fd}_{\psi,n}(s)$};
\draw[red] [-, thick] (0,6) -- (0.25,6);
\draw[red][thick] plot[domain=0.5:1, smooth] (\x,{6-2*\x});
\draw[red] [-, thick] (1,2) -- (2,2);
\draw[green][thick]plot[domain=0.25:0.5, smooth] (\x,{6-2*\x});
\draw[cyan][thick]plot[domain=0.25:0.5, smooth] (\x,{7-4*\x});
\draw [fill] (1,2) circle [radius=.05] node [below] {\scriptsize$\frac{1}{2}$};
\draw [fill] (0.5,2) circle [radius=.05] node [below] {\scriptsize$\frac{1}{4}$};
\draw [fill] (0.25,2) circle [radius=.05] node [below] {\scriptsize$\frac{a}{4}$};
\draw [fill] (0,6) circle [radius=.05] node [left] {\scriptsize$n$};
\draw [fill] (0,4) circle [radius=.05] node [left] {\scriptsize$n-1$};
\draw [fill] (0,5) circle [radius=.05] node [left] {\scriptsize$n-\frac{1}{2}$};
\draw [fill] (0,5.5) circle [radius=.05] node [left] {\scriptsize$n-\frac{a}{2}$};
\draw[gray] [-, dashed] (0,5.5) -- (0.25,5.5);
\draw[gray] [-, dashed] (0.25,2) -- (0.25,6);
\draw[gray] [-, dashed] (0.5,2) -- (0.5,5);
\draw[gray] [-, dashed] (1,2) -- (1,4);
\draw[gray] [-, dashed] (0,4) -- (1,4);
\draw[gray] [-, dashed] (0,5) -- (0.5,5);
\draw[red] [-, thick] (1.5,6) -- (1.7,6);
\pgftext[at={\pgfpoint{3cm}{6cm}}]{\textcolor{red}{sharp result}};
\draw[cyan] [-, thick] (1.5,5.5) -- (1.7,5.5);
\pgftext[at={\pgfpoint{3cm}{5.5cm}}]{\textcolor{cyan}{upper bound}};	
\draw[green] [-, thick] (1.5,5) -- (1.7,5);
\pgftext[at={\pgfpoint{3cm}{5cm}}]{\textcolor{green}{lower bound}};	
\end{tikzpicture}
    \caption{$a \in (0,1)$}
    \label{radialfigure1.2}
  \end{subfigure}
  \hfill 
  \begin{subfigure}{0.3\textwidth}
    \centering
    \begin{tikzpicture}
\draw [->, thick] (0,2) -- (2,2) node [below left] {$s$};
\draw [->, thick] (0,2) -- (0,7)  node [below left] {$\wtilde{\fd}_{\psi,n}(s)$};
\draw[red] [-, thick] (0,6) -- (1,6);
\draw[red] [-, thick] (1,2) -- (2,2);
\draw [fill] (1,2) circle [radius=.05] node [below] {\scriptsize$\frac{1}{2}$};
\draw [fill] (0,6) circle [radius=.05] node [left] {\scriptsize$n$};
\draw[gray] [-, dashed] (1,2) -- (1,6);
\end{tikzpicture}
    \caption{$a =1$}
    \label{radialfigure1.3}
  \end{subfigure}
  \caption{$\wtilde{\fd}_{\psi,n}$ for fractional Schr\"odinger, $\psi(\|\xi\|)=\|\xi\|^a$}
  \label{radialfigure1}
\end{figure}
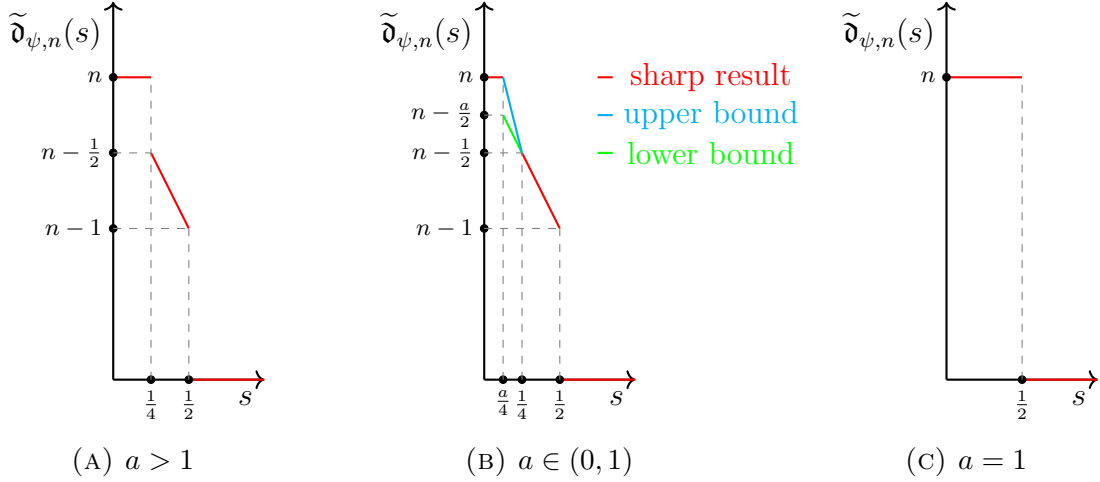

Some remarks are now in order.

\begin{remark}
The result given by (\ref{BR_estimates}), due to Bennett-Rogers \cite{BR}, for the fractional Schr\"odinger equations with convex phase and radial initial data, refined the result of Prestini \cite{Prestini} for the classical setting of the Lebesgue measure on $\R^n$. In the same spirit, Theorem \ref{radial_concave_thm} refines the results of Walther \cite{Walther1, Walther2, Walther3} for the concave phases and the Wave propagation.

\end{remark}

\begin{remark}
In Subsection \ref{subsec6.3}, as applications of a transference principle (Proposition \ref{radial_transference_proposition}), we further our study as
\begin{enumerate}
\item  we obtain Hausdorff dimension estimates of divergence sets for families of abstract dispersive equations, with radial initial data (Corollaries \ref{radial_corollary1}, \ref{radial_corollary2}, and \ref{radial_corollary3}) which can be viewed as variations of the results obtained in Section \ref{sec5};
\item as special cases,
\begin{enumerate}
\item we recover the known result for the Beam propagation (Remark \ref{radial_beam_remark});
\item we also obtain genuinely new results for the linearized Boussinesq (Remark \ref{radial_beam_remark}) and the Klein-Gordon propagations (Remark \ref{radial_klein_remark}).
\end{enumerate}
\end{enumerate}
\end{remark}

We now proceed with our departure from $\R^n$ to the setting of the real hyperbolic spaces. As mentioned at the very beginning, resolving the Carleson's conjecture for the Schr\"odinger equation on $\R^n$ required advanced tools from decoupling theory. In the case of a Riemannian manifold, where the volume measure neither satisfies any doubling property nor does it admit dilations, the above mentioned Euclidean machineries work no longer and thus it offers a fresh challenge. The prototypical examples of such manifolds are the real hyperbolic spaces $\mathbb{H}^n$, which due to their negative curvature, are of purely exponential volume growth. Inspite of recent attempts and some partial results \cite{WZ, Dewan4}, a complete resolution of the Carleson's conjecture on $\mathbb{H}^n$ seems to be a distant prospect.

In the special case of geodesically radial (or simply radial) initial data, however, Carleson's conjecture for various dispersive PDEs has recently been settled in a series of papers by the author \cite{Dewan1, DR, Dewan2, Dewan3}. We now focus on the prototypical setting of the three dimensional real hyperbolic space $\Hb^3$ by considering \begin{equation*}
\wtilde{\fd}_{\psi,\Hb^3}(s):= \displaystyle\sup_{\substack{f \in H^s(\Hb^3) \\ f \:\:\text{ radial }}} dim_{\cH}\D_{\psi}(f)\:,
\end{equation*}
where 
\begin{equation*}
\D_{\psi}(f):=\left\{x \in \Hb^3 : \displaystyle\lim_{t \to 0} e^{it\psi\left(\sqrt{-\Delta_{\Hb^3}}\right)}f(x) \ne f(x)\right\}\:,\:\: f \in H^s(\Hb^3)\:,
\end{equation*}
and $\Delta_{\Hb^3}$ is the Laplace-Beltrami operator on $\Hb^3$. Our next result is the desired fractal refinement:
\begin{theorem} \label{H3_theorem}
For radial initial data on $\Hb^3$, we have
\begin{enumerate}
\item for the Schr\"odinger equation, the fractional Schr\"odinger equations with convex phase, the Beam and the linearized Boussinesq equation,
\begin{equation*} 
\wtilde{\fd}_{\psi,\Hb^3}(s)= \begin{cases}
3,\:\: &s<\frac{1}{4}\\
3-2s,\:\: &\frac{1}{4} \le s<\frac{1}{2}\\
0,\:\: &\frac{1}{2} < s\:;
\end{cases}
\end{equation*}  
\item for the fractional Schr\"odinger equations with concave phase,  i.e. $a \in (0,1)$,
\begin{enumerate}
\item the sharp result in the range,
\begin{equation*}
\wtilde{\fd}_{\psi,\Hb^3}(s) =\begin{cases}
3\:,\:\: &s<\frac{a}{4}\\
3-2s\:,\:\: &\frac{1}{4} \le s<\frac{1}{2}\\
0\:,\:\: &\frac{1}{2} < s\:;
\end{cases}
\end{equation*}
\item and in the remaining range, \\
the upper bound,
\begin{equation*}
\wtilde{\fd}_{\psi,\Hb^3}(s) \le \frac{5}{2}+\frac{1-4s}{2(1-a)}\:,\:\: \frac{a}{4} \le s \le \frac{1}{4}\:,
\end{equation*}
and the lower bound,
\begin{equation*}
\wtilde{\fd}_{\psi,\Hb^3}(s) \ge 3-2s\:,\:\: \frac{a}{4} \le s \le \frac{1}{4}\:;
\end{equation*}
\end{enumerate}
\item for the Wave equation and the Klein-Gordon equations, we have
\begin{equation*} 
\wtilde{\fd}_{\psi,\Hb^3}(s)= \begin{cases}
3,\:\: &s \le \frac{1}{2}\\
0,\:\: &\frac{1}{2} < s\:.
\end{cases}
\end{equation*}
\end{enumerate}
\end{theorem}

The results in Theorem \ref{H3_theorem} are presented in Figure \ref{H3figure}.

\begin{figure}[htbp]
  \centering
  
  \begin{subfigure}{0.3\textwidth}
    \centering
    \begin{tikzpicture}
\draw [->, thick] (0,2) -- (2,2) node [below left] {$s$};
\draw [->, thick] (0,2) -- (0,7)  node [below left] {$\wtilde{\fd}_{\psi,\Hb^3}(s)$};
\draw[red] [-, thick] (0,6) -- (0.5,6);
\draw[red][thick] plot[domain=0.5:1, smooth] (\x,{6-2*\x});
\draw[red] [-, thick] (1,2) -- (2,2);
\draw [fill] (1,2) circle [radius=.05] node [below] {\scriptsize$\frac{1}{2}$};
\draw [fill] (0.5,2) circle [radius=.05] node [below] {\scriptsize$\frac{1}{4}$};
\draw [fill] (0,6) circle [radius=.05] node [left] {\scriptsize$3$};
\draw [fill] (0,5) circle [radius=.05] node [left] {\scriptsize$\frac{5}{2}$};
\draw [fill] (0,4) circle [radius=.05] node [left] {\scriptsize$2$};
\draw[gray] [-, dashed] (1,2) -- (1,4);
\draw[gray] [-, dashed] (0.5,2) -- (0.5,6);
\draw[gray] [-, dashed] (0,4) -- (1,4);
\draw[gray] [-, dashed] (0,5) -- (0.5,5);
\end{tikzpicture}
    \caption{PDEs in part (1)}
    \label{H3figure1}
  \end{subfigure}
  \hfill 
  \begin{subfigure}{0.3\textwidth}
    \centering
    \begin{tikzpicture}
\draw [->, thick] (0,2) -- (2,2) node [below left] {$s$};
\draw [->, thick] (0,2) -- (0,7)  node [below left] {$\wtilde{\fd}_{\psi,\Hb^3}(s)$};
\draw[red] [-, thick] (0,6) -- (0.25,6);
\draw[red][thick] plot[domain=0.5:1, smooth] (\x,{6-2*\x});
\draw[red] [-, thick] (1,2) -- (2,2);
\draw[cyan][thick]plot[domain=0.25:0.5, smooth] (\x,{7-4*\x});
\draw[cyan][green]plot[domain=0.25:0.5, smooth] (\x,{6-2*\x});
\draw [fill] (1,2) circle [radius=.05] node [below] {\scriptsize$\frac{1}{2}$};
\draw [fill] (0.5,2) circle [radius=.05] node [below] {\scriptsize$\frac{1}{4}$};
\draw [fill] (0.25,2) circle [radius=.05] node [below] {\scriptsize$\frac{a}{4}$};
\draw [fill] (0,6) circle [radius=.05] node [left] {\scriptsize$3$};
\draw [fill] (0,4) circle [radius=.05] node [left] {\scriptsize$2$};
\draw [fill] (0,5) circle [radius=.05] node [left] {\scriptsize$\frac{5}{2}$};
\draw [fill] (0,5.5) circle [radius=.05] node [left] {\scriptsize$3-\frac{a}{2}$};
\draw[gray] [-, dashed] (0,5.5) -- (0.25,5.5);
\draw[gray] [-, dashed] (0.25,2) -- (0.25,6);
\draw[gray] [-, dashed] (0.5,2) -- (0.5,5);
\draw[gray] [-, dashed] (1,2) -- (1,4);
\draw[gray] [-, dashed] (0,4) -- (1,4);
\draw[gray] [-, dashed] (0,5) -- (0.5,5);
\draw[red] [-, thick] (1.5,6) -- (1.7,6);
\pgftext[at={\pgfpoint{3cm}{6cm}}]{\textcolor{red}{sharp result}};
\draw[cyan] [-, thick] (1.5,5.5) -- (1.7,5.5);
\pgftext[at={\pgfpoint{3cm}{5.5cm}}]{\textcolor{cyan}{upper bound}};	
\draw[cyan] [-, thick] (1.5,5) -- (1.7,5);
\pgftext[at={\pgfpoint{3cm}{5.cm}}]{\textcolor{green}{lower bound}};	
\end{tikzpicture}
    \caption{PDEs in part (2)}
    \label{H3figure2}
  \end{subfigure}
  \hfill 
  \begin{subfigure}{0.3\textwidth}
    \centering
    \begin{tikzpicture}
\draw [->, thick] (0,2) -- (2,2) node [below left] {$s$};
\draw [->, thick] (0,2) -- (0,7)  node [below left] {$\wtilde{\fd}_{\psi,\Hb^3}(s)$};
\draw[red] [-, thick] (0,6) -- (1,6);
\draw[red] [-, thick] (1,2) -- (2,2);
\draw [fill] (1,2) circle [radius=.05] node [below] {\scriptsize$\frac{1}{2}$};
\draw [fill] (0,6) circle [radius=.05] node [left] {\scriptsize$3$};
\draw[gray] [-, dashed] (1,2) -- (1,6);
\end{tikzpicture}
    \caption{PDEs in part (3)}
    \label{H3figure3}
  \end{subfigure}
  \caption{$\wtilde{\fd}_{\psi,\Hb^3}$ for dispersive PDEs in Theorem \ref{H3_theorem} on $\Hb^3$}
  \label{H3figure}
\end{figure}

\begin{remark}
\begin{enumerate}
\item Theorem \ref{H3_theorem} is the first result of its kind, outside the setting of flat Euclidean geometry.
\item In particular, coupled with \cite{Cowling}, it resolves the classical Carleson's problem for the Wave equation for the Riemannian volume measure on $\Hb^3$.  
\end{enumerate}
\end{remark}

We now briefly summarize the key ideas and novelties of this article:
\begin{itemize}
\item The proof of Theorem \ref{transference_theorem} consists of the following steps:


\begin{itemize}
\item First, starting from the estimate on $\fd_{\psi_1,n}$, by applying the Frostmann's lemma, we get a pointwise convergence result to the initial data for the $\psi_1$-dispersive propagation, beyond a fractal threshold. Then by the dynamics of the $\psi_1$-dispersive evolution group, we propagate the pointwise convergence phenomenon for all time. Combining this with the fixed-time fractal Sobolev embedding (Corollary \ref{fractal-Sobolev-embedding}), we obtain a.e. finiteness of the maximal function, with respect to the fractal measures. Then appealing to Nikishin theory, we obtain weighted, weak, fractal maximal estimates for the $\psi_1$-dispersive propagation. 


\item Then we transfer the above weighted, weak, fractal maximal estimates for the $\psi_1$-dispersive propagation to its $\psi_2$-counterpart. In order to get this, we work with Schwartz class initial data and focus separately on small and large frequencies. The former is easier to handle and follows from a simple application of the Cauchy-Schwarz inequality. The latter, however, is more interesting and is estimated by making a Taylor series expansion of the Fourier multipliers in time and utilizing the property of comparable oscillation.


\item Finally, by the standard method of maximal functions, fractal Sobolev embedding for the initial data (Corollary \ref{fractal-Sobolev-embedding}) and Frostmann's lemma, we get the subordination relation $\fd_{\psi_2,n} \le \fd_{\psi_1,n}$. By the symmetry of the argument, the result then follows.
\end{itemize}


\item  As applications of Theorem \ref{transference_theorem}, we obtain Hausdorff dimension estimates of divergence sets for families of abstract dispersive equations, classified solely in terms of the oscillation of their Fourier multipliers in high frequency (Corollaries \ref{corollary1}, \ref{corollary2}, \ref{corollary3} and \ref{corollary4}). As special cases, we recover the known one-dimensional results for the linearized Boussinesq and the Beam propagations (Remark \ref{recovering_results_beam}), by means of a simple one-line proof, compared to the oscillatory integral estimates involving higher derivatives in \cite{LiLi,LLN}. We also obtain genuinely new results for the linearized Boussinesq and the Beam propagations in higher order dimensions (Remark \ref{new_results_beam}) and the Klein-Gordon propagations in all dimensions (Remark \ref{new_results_klein}).


\item The geometric input of radial symmetry brings the analytic ingredient of the Bessel functions into the mix. In view of this, the proof of part (1) of Theorem \ref{radial_concave_thm}, requires consideration of three different regularity regimes:
\begin{itemize}
\item for low regularity i.e. $s<a/4$, the result follows by Walther's work in \cite{Walther3};
\item for high regularity i.e. $s>1/2$, the result follows by the pointwise decay of the Bessel functions (Lemma \ref{lemma_s>1/2});
\item in between i.e. $s \in [a/4,1/2)$, the upper bound is obtained by efficiently estimating the joint oscillation of the Bessel function and the Fourier multiplier. This is where the concavity of the phase becomes crucial. The lower bound in this range follows from the size estimates of the singular sets of radial Sobolev functions (Remark \ref{BR_remark}).
\end{itemize}


\item As applications of a transference principle (Proposition \ref{radial_transference_proposition}) specialized to radial initial data, we obtain Hausdorff dimension estimates of divergence sets for families of abstract dispersive equations, with radial initial data (Corollaries \ref{radial_corollary1}, \ref{radial_corollary2}, and \ref{radial_corollary3}) which can be viewed as variations of the results obtained in Section \ref{sec5}. As special cases, we recover the known result for the Beam propagation (Remark \ref{radial_beam_remark}), by means of a simple one-line proof, compared to the oscillatory integral estimates involving higher order derivatives in \cite{LLN}. We also obtain genuinely new results for the linearized Boussinesq (Remark \ref{radial_beam_remark}) and the Klein-Gordon propagations (Remark \ref{radial_klein_remark}).
\item In Section \ref{sec7}, we study the dimensions of divergence sets for dispersive equations in the non-Euclidean setting of $\Hb^3$, which is the first such attempt known to the author. Our approach does not imitate the Euclidean arguments for $\Hb^3$, but rather is built upon the following correspondence results:
\begin{itemize}
\item a one-to-one correspondence between the $\alpha$-dimensional fractal measures on geodesic annuli in a non-positively curved Hadamard manifold and the ones on annuli in the tangent space (with respect to the flat metric), induced by the Riemannian exponential map (Lemma \ref{fractal_measure_correspondence_lemma});
\item a one-to-one correspondence between the classes of radial Schwartz class functions on $\Hb^3$ and $\R^3$ satisfying identities in terms of the Fourier transforms  and the Sobolev norms. This is obtained via the Abel transforms on $\R^3$ and on Ch\'ebli-Trim\`eche hypergroups, leveraging the underlying Riemannian symmetric space structure of $\Hb^3$ (Lemma \ref{schwartz_correspondence_lemma}).
\end{itemize}
\item Using the above correspondence results, we establish a connection between dispersive propagations on $\Hb^3$ and $\R^3$ with suitable Fourier multipliers, to obtain the fixed-time fractal Sobolev embedding (Lemma \ref{H3-Sobolev-embedding}) and a maximal estimate in high regularity $s>1/2$ (Lemma \ref{H3-lemma_s>1/2}) on $\Hb^3$.
\item All of the ideas mentioned in this article until now, culminate in the form of an inter-spatial transference principle (Theorem \ref{H3_transference}) which allows us to conclude about the Hausdorff dimensions of the divergence sets of dispersive equations on $\Hb^3$ from their counterpart on $\R^3$. This is the heart of the matter in proving Theorem \ref{H3_theorem}.

\item In the process, we also conduct a quantitative study of the size of singular sets of Sobolev functions on $\Hb^3$ (Theorem \ref{sobolev_singularity_theorem}), in the spirit of \v{Z}ubrini\'c \cite{Zubrinic}. In order to do so, we also obtain the explicit Bessel kernel on $\Hb^3$ via the heat semi-group representation of the fractional Laplace-Beltrami operator (Lemma \ref{lemma_bessel_kernel}). These results may be of independent interest.

\end{itemize}

We end this commentary with the layout of this article. In Section \ref{sec2}, we fix the notations. In Section \ref{sec3}, we explicitly detail the underlying technical set-up. In Section \ref{sec4}, we prove Theorem \ref{transference_theorem} and present some of its applications in Section \ref{sec5}. In Section \ref{sec6}, we prove Theorem \ref{radial_concave_thm} and also results for more general dispersive families with radial initial data. In Section \ref{sec7}, we prove Theorem \ref{H3_theorem} and the other results on $\Hb^3$. Finally, in Section \ref{sec8}, we conclude by making some remarks and posing some new problems.

\section{Notations}
\label{sec2}
In this section, we fix the notations of this article. $\N$ will denote the set of positive integers. For $x \in \R$, $\lfloor x\rfloor$ ($\lceil x\rceil$ resp.) will denote the largest integer smaller than or equal to $x$ (the smallest integer greater than or equal to $x$ resp.). $\I$ will be used to denote the interval $[-1,1]$, whereas $0$ will be used to denote both the number as well as the vector $(0,\cdots,0) \in \R^n$ and will be clear from the context. For $x \in \R^n$, $\|x\|$ will denote the Euclidean norm.
$\mathbb{B}(x,r)$ will denote the ball in $\R^n$ with center $x$ and radius $r$ and $\mathbb{A}(r_1,r_2)$ will denote the annulus in $\R^n$ with center $0$, inner radius $r_1$ and outer radius $r_2$. For a set $A$, $\chi_A$ will be used to denote the indicator function of $A$. For a metric space $X,\:\mathscr{B}(X)$ will denote the collection of Borel subsets of $X$. `C' will denote positive constants whose values may change on each occurrence. For two non-negative functions $f_1,f_2$, $f_1 \lesssim f_2$ ( $f_1 \gtrsim f_2$ resp.) will denote $f_1 \le Cf_2$ ($f_1 \ge (1/C)f_2$ resp.) for some $C\ge 1$. The notation $f_1 \asymp f_2$ will be used if both $f_1 \lesssim f_2$ and $f_1 \gtrsim f_2$ are true. We use the notation $f_1 \gg f_2$ if there exists a large constant $M>0$ such that $f_1 \ge Mf_2$. For complex valued $f_1$, we write $f_1=\mathcal{O}(f_2)$ if $|f_1| \lesssim f_2$.

\section{Some technical nitty-gritty}
\label{sec3}
A priori, elements of $H^s(\R^n)$ are just equivalence classes and such different representations can differ on sets of full Hausdorff dimensions. Thus to make the notion of the Hausdorff dimension of divergence sets well-defined, one must fix a canonical choice of each such equivalence class in $H^s(\R^n)$, which is also well-defined with respect to fractal measures. In this section, we make this precise by following \cite{BBCR}.

For initial data $f$ belonging to the Schwartz class $\Ss(\R^n)$, the solution of the dispersive equation (\ref{dispersive}),  expressed pointwise by the formula,
\begin{equation} \label{schwartz_solution}
u(x,t) = \frac{1}{(2\pi)^n}\int_{\R^n} e^{i\left(x\cdot \xi + t\psi(\|\xi\|)\right)}\:\what{f}(\xi)\:d\xi\:,
\end{equation} 
is continuous and hence there are no issues. But for a general $f \in H^s(\R^n)$, the integral in (\ref{schwartz_solution}) need not exist. In this scenario, taking $\Phi$ to be the Gaussian, i.e. $\Phi(\xi):=e^{-\|\xi\|^2}$, we consider the operators,
\begin{equation} \label{N-scale}
S^N_{\psi,t}f(x):= \frac{1}{(2\pi)^n}\int_{\R^n} \Phi\left(N^{-1}\xi\right)\:e^{i\left(x\cdot \xi + t\psi(\|\xi\|)\right)}\:\what{f}(\xi)\:d\xi\:,\:\:\:N \in \N\:,
\end{equation}  
and then define $u$ as the pointwise limit (whenever it exists),
\begin{equation} \label{canonical_representative}
u(\cdot,t):=\lim_{N \to \infty} S^N_{\psi,t}f\:.
\end{equation}
It is easy to see that, the choice of $u(\cdot,t)$ given by (\ref{canonical_representative}) coincides with the classical $L^2$-limits, in the sense of Lebesgue. But interestingly, it is also well-defined with respect to fractal measures for any $s>0$.

We now recall fractal measures. Let $B$ be a Borel subset of $\R^n$. For $\alpha \in [0,n]$, we denote by $\M^\alpha(B)$ the collection of {\em $\alpha$-dimensional measures} on $B$, that is, positive Radon measures $\mu$ with $Supp(\mu) \subset B$ and
\begin{equation} \label{Frostamnn_constant}
c_\alpha(\mu):= \displaystyle\sup_{x \in \R^n,r>0} \frac{\mu(\mathbb{B}(x,r))}{r^\alpha} < \infty\:.
\end{equation}  
The quantity $c_\alpha(\mu)$ is called the {\em $\alpha$-Frostmann constant} of $\mu$. For more details, we refer to \cite{Mattila}.

The key estimate in proving the fractal well-definedness of the choice (\ref{canonical_representative}) is the following, which essentially follows from the arguments presented in Appendices A and C in \cite{BBCR}:
\begin{lemma} \label{N-maximal-estimate}
Let $s>0$ and $\alpha >n-2s$. Then for any $f \in H^s(\R^n)$, any $\mu \in \M^\alpha(\mathbb{B}(0,1))$ and any $t \in \R$, we have
\begin{equation*}
\left\|\sup_{N \ge 1} \left|S^N_{\psi,t}f\right|\right\|_{L^2(d\mu)} \lesssim \sqrt{c_\alpha(\mu)}\:\|f\|_{H^s(\R^n)}\:.
\end{equation*}
\end{lemma}

A straight-forward consequence of Lemma \ref{N-maximal-estimate} is the following fractal Sobolev embedding for the choice (\ref{canonical_representative}), which garners the desired well-definedness on the fractal level:
\begin{corollary} \label{fractal-Sobolev-embedding}
Let $s>0$ and $\alpha >n-2s$. Then for any $f \in H^s(\R^n)$, any $\mu \in \M^\alpha(\mathbb{B}(0,1))$ and any $t \in \R$,  we have
\begin{equation*}
\left\|u(\cdot,t)\right\|_{L^2(d\mu)} \lesssim \sqrt{c_\alpha(\mu)}\:\|f\|_{H^s(\R^n)}\:.
\end{equation*}
\end{corollary}
Henceforth in this article, the solution will always be represented by the canonical choice (\ref{canonical_representative}). By Appendix A in \cite{BBCR}, we also have that the choice (\ref{canonical_representative}) for the initial data $f$ also coincides up to a set of Hausdorff dimension $\le n-2s$, with the classical choice given as a Bessel potential. 

\begin{remark} \label{Zubrinic_remark}
The index $n-2s$ showing up above, is actually the critical index for the Hausdorff dimension of singular sets of Sobolev functions in $H^s(\R^n)$ for $s \in (0,n/2)$. In \cite{Zubrinic}, it was shown that the Bessel potential representation can be singular on sets with dimension smaller than this. Consequently, our problem of estimating Hausdorff dimension of divergence sets become ill-posed for $\fd_{\psi,n}(s)<n-2s$.
\end{remark}  

\section{The transference principle: proof of Theorem \ref{transference_theorem}}
\label{sec4}
\subsection*{Step 1: size estimates of divergence sets to weighted, weak, fractal maximal estimates} By Corollary \ref{fractal-Sobolev-embedding} and Remark \ref{Zubrinic_remark}, $\fd_{\psi_1,n}$ is well-defined for all $s>0$, with $\fd_{\psi_1,n}(s) \ge n-2s$ (with $\fd_{\psi_1,n}(s) = 0$ for all $s>n/2$ by the Sobolev embedding). Let $s_0>0$. Then by the definition, 
\begin{equation*}
dim_{\cH}\D_{\psi_1}(f) \le \fd_{\psi_1,n}(s_0)\:,\:\:\text{ for all } f \in H^{s_0}(\R^n)\:.
\end{equation*}  
So in particular,
\begin{equation*}
dim_{\cH}\left(\D_{\psi_1}(f) \cap \mathbb{B}(0,1)\right) \le \fd_{\psi_1,n}(s_0)\:,\:\:\text{ for all } f \in H^{s_0}(\R^n)\:.
\end{equation*} 
Thus given $\varepsilon>0$, we have for $\alpha:=\fd_{\psi_1,n}(s_0)+\varepsilon$,
\begin{equation*}
\cH^\alpha\left(\D_{\psi_1}(f) \cap \mathbb{B}(0,1)\right)=0\:,\:\:\text{ for all } f \in H^{s_0}(\R^n)\:,
\end{equation*}
where $\mathcal{H}^\alpha$ is the $\alpha$-dimensional Hausdorff measure. Then by Frostmann's lemma,
\begin{equation*}
\mu\left(\D_{\psi_1}(f) \cap \mathbb{B}(0,1)\right)=0\:,\:\:\text{ for all } \mu \in \M^\alpha(\mathbb{B}(0,1))\:,\:\:\text{ for all } f \in H^{s_0}(\R^n)\:,
\end{equation*}
where $\M^\alpha(\mathbb{B}(0,1))$ is as defined in (\ref{Frostamnn_constant}). This means that for all $\mu \in \M^\alpha(\mathbb{B}(0,1))$ and for all $f \in H^{s_0}(\R^n)$,
\begin{equation} \label{main_thm_pf_eq1}
e^{it\psi_1(\sqrt{-\Delta})} f(x) \to f(x)\:,\:\text{ as } t \to 0\:,\: \mu-\text{a.e. } x \in \mathbb{B}(0,1)\:\:.
\end{equation}
We now fix $f \in H^{s_0}(\R^n),\mu \in \M^\alpha(\mathbb{B}(0,1)),t_0 \in \I:=[-1,1]$ and set $g:=e^{it_0\psi_1(\sqrt{-\Delta})}f$. Then as $\left\{e^{it\psi_1(\sqrt{-\Delta})}\right\}_{t \in \R}$ is a group of unitary operators on $H^{s_0}(\R^n)$, we have that $g \in H^{s_0}(\R^n)$. Therefore (\ref{main_thm_pf_eq1}) implies that 
\begin{equation*}
e^{ih\psi_1(\sqrt{-\Delta})} g(x) \to g(x)\:,\:\text{ as } h \to 0\:,\: \mu-\text{a.e. } x \in \mathbb{B}(0,1)\:\:.
\end{equation*}
Now using the group property, we note that
\begin{equation*}
e^{ih\psi_1(\sqrt{-\Delta})} g= e^{ih\psi_1(\sqrt{-\Delta})} \left(e^{it_0\psi_1(\sqrt{-\Delta})}f\right)=  e^{i(t_0+h)\psi_1(\sqrt{-\Delta})} f\:,
\end{equation*}
and hence,
\begin{equation} \label{main_thm_pf_eq2}
e^{it\psi_1(\sqrt{-\Delta})} f(x) \to e^{it_0\psi_1(\sqrt{-\Delta})} f(x)\:,\:\text{ as } t \to t_0\:,\: \mu-\text{a.e. } x \in \mathbb{B}(0,1)\:\:.
\end{equation}
Next for each $q \in \Q \cap \I$, let
\begin{eqnarray*}
\mathcal{N}^{(1)}_{q,f}&:=& \left\{x \in \mathbb{B}(0,1):  \left|e^{iq\psi_1(\sqrt{-\Delta})} f(x)\right| < \infty\right\}\:,\\
\mathcal{N}^{(2)}_{q,f}&:=& \left\{x \in \mathbb{B}(0,1): \lim_{t \to q} e^{it\psi_1(\sqrt{-\Delta})} f(x) \ne e^{iq\psi_1(\sqrt{-\Delta})} f(x)\right\}\:.
\end{eqnarray*}
Then by the fixed-time fractal Sobolev embedding (Corollary \ref{fractal-Sobolev-embedding}), we have
\begin{equation*}
\mu\left(\left(\mathcal{N}^{(1)}_{q,f}\right)^c\right)=0\:,\:\:\text{ for all } q \in \Q \cap \I\:.
\end{equation*}
Also by (\ref{main_thm_pf_eq2}), it follows that
\begin{equation*}
\mu \left(\mathcal{N}^{(2)}_{q,f}\right)=0\:,\:\:\text{ for all } q \in \Q \cap \I\:.
\end{equation*}
Then setting
\begin{equation*}
\mathcal{N}_f:= \bigcup_{q \in \Q \cap \I} \left(\left(\mathcal{N}^{(1)}_{q,f}\right)^c \cup \mathcal{N}^{(2)}_{q,f}\right)\:,
\end{equation*}
we note that $\mu\left(\mathcal{N}_f\right)=0$. Also, for all $x \in \mathbb{B}(0,1) \setminus \mathcal{N}_f$ and for all $q \in \Q \cap \I$,
\begin{equation} \label{main_thm_pf_eq3}
\left|e^{iq\psi_1(\sqrt{-\Delta})} f(x)\right| < \infty \:\:\text{ and } e^{it\psi_1(\sqrt{-\Delta})} f(x) \to e^{iq\psi_1(\sqrt{-\Delta})} f(x)\:\text{ as } t \to q \:.
\end{equation}
Next, we fix $x \in \mathbb{B}(0,1) \setminus \mathcal{N}_f$ and $q \in \Q \cap \I$. By (\ref{main_thm_pf_eq3}), there exists $\nu_{q,f}(x)>0$ such that for $t \in \left(q-\nu_{q,f}(x)\:,\:q+\nu_{q,f}(x)\right)$, we have
\begin{equation*}
\left|e^{it\psi_1(\sqrt{-\Delta})} f(x) - e^{iq\psi_1(\sqrt{-\Delta})} f(x)\right| <1\:,
\end{equation*}
and hence,
\begin{equation*}
\left|e^{it\psi_1(\sqrt{-\Delta})} f(x)\right| < \left|e^{iq\psi_1(\sqrt{-\Delta})} f(x)\right|\:+\:1\:. 
\end{equation*}
Consequently, we get
\begin{equation} \label{main_thm_pf_eq4}
\sup_{t \in \left(q-\nu_{q,f}(x)\:,\:q+\nu_{q,f}(x)\right)}\left|e^{it\psi_1(\sqrt{-\Delta})} f(x)\right| < \infty\:.
\end{equation}
Now by compactness of $\I$, there exist finitely many $q_1,q_2,\dots,q_N \in \Q \cap \I$ such that
\begin{equation*}
\I \subset \bigcup_{j=1}^N \left(q_j-\nu_{q_j,f}(x)\:,\:q_j+\nu_{q_j,f}(x)\right)\:.
\end{equation*}
Hence by (\ref{main_thm_pf_eq4}), we get that the maximal function
\begin{eqnarray*}
S^*_{\psi_1}f(x):= \sup_{t \in \I} \left|e^{it\psi_1(\sqrt{-\Delta})} f(x)\right| \le \max_{1 \le j \le N}  \sup_{t \in \left(q_j-\nu_{q_j,f}(x)\:,\:q_j+\nu_{q_j,f}(x)\right)}\left|e^{it\psi_1(\sqrt{-\Delta})} f(x)\right| < \infty\:,
\end{eqnarray*}
for all $x \in \mathbb{B}(0,1) \setminus \mathcal{N}_f$. Therefore, 
\begin{equation*}
S^*_{\psi_1}f < \infty\:,\:\:\mu-\text{a.e. },\:\text{ for all } f \in H^{s_0}(\R^n)\:.
\end{equation*}
Then by Nikishin theory \cite[Proposition 1.4 and Corollary 2.7, Chapter VI]{Rubio}, there exists a $\mu$-measurable weight $w$ with $w>0$, $\mu$-a.e. on $\mathbb{B}(0,1)$ such that we have the weighted weak-$L^2$ fractal maximal estimate:
\begin{equation} \label{general_inequality}
\left\|S^*_{\psi_1}f\right\|_{L^{2,\infty}(w\:d\mu)} \lesssim \|f\|_{H^{s_0}(\R^n)}\:,\:\:\text{ for all } f \in H^{s_0}(\R^n)\:.
\end{equation}
Thus in particular, we have
\begin{equation} \label{main_thm_pf_eq5}
\left\|S^*_{\psi_1}f\right\|_{L^{2,\infty}(w\:d\mu)} \lesssim \|f\|_{H^{s_0}(\R^n)}\:,\:\:\text{ for all } f \in \Ss(\R^n)\:.
\end{equation}
By replacing $w$ with $\min\{w,1\}$, we assume that $w \le 1,\:\mu$-a.e., in (\ref{general_inequality}) and (\ref{main_thm_pf_eq5}).

\subsection*{Step 2: transference of weighted, weak, fractal maximal estimates} We now fix $f \in \Ss(\R^n)$. By the hypothesis on the phases, there exist $\Lambda>0$ and $C>0$, such that  
\begin{equation} \label{main_thm_pf_eq6}
\left|\psi_1(\|\xi\|)-\psi_2(\|\xi\|)\right| \le C\:, \text{ for all } \|\xi\| > \Lambda\:.
\end{equation}  
Next, we consider an auxiliary function $\eta \in C^\infty_c(\R^n)$ radial, non-negative such that 
\begin{equation*}
Supp(\eta) \subset \left\{\xi \in \R^n : \frac{1}{2} < \|\xi\| < 2\right\}\:\:\text{ and } \sum_{k=-\infty}^\infty \eta\left(2^{-k}\xi\right)=1\:,\:\text{ for } \xi \ne 0\:.
\end{equation*}
Let $k_o$ be the largest integer $k$ such that $k \le \log_2\lceil\Lambda\rceil+1$. Then setting,
\begin{equation*}
\eta_1(\xi):=\sum_{k=-\infty}^{k_0} \eta\left(2^{-k}\xi\right)\:,\text{ and } \eta_2(\xi):=\sum_{k=k_0+1}^\infty \eta\left(2^{-k}\xi\right)\:,
\end{equation*}
we get that
\begin{equation*}
\what{f} \equiv \what{f}\cdot\eta_1\:+\:\what{f}\cdot\eta_2\:.
\end{equation*}
Thus for $m=1,2$, there exists $f_m\in \Ss(\R^n)$ such that $\what{f}_m=\what{f}\cdot\eta_m$ and hence,
\begin{equation} \label{main_thm_pf_eq7}
\what{f} \equiv \what{f}_1 \:+\: \what{f}_2\:,
\end{equation}
with
\begin{equation*}
Supp\left(\what{f}_m\right) \subset \begin{cases}
\mathbb{B}\left(0,2^{k_0+1}\right)\:\:&\text{ if } m=1\:,\\
\R^n \setminus \mathbb{B}\left(0,2^{k_0}\right)\:\:&\text{ if } m=2\:.
\end{cases}
\end{equation*}
Thus by the quasi-triangle inequality,
\begin{equation} \label{main_thm_pf_eq8}
\left\|S^*_{\psi_2}f\right\|_{L^{2,\infty}(w\:d\mu)} \lesssim \left\|S^*_{\psi_2}f_1\right\|_{L^{2,\infty}(w\:d\mu)}\:+\: \left\|S^*_{\psi_2}f_2\right\|_{L^{2,\infty}(w\:d\mu)}\:,
\end{equation}
where we recall that $S^*_{\psi_2}f$ is the maximal function,
$$S^*_{\psi_2}f(x):= \sup_{t \in \I} \left|e^{it\psi_2(\sqrt{-\Delta})} f(x)\right|\:.$$
For small frequency, using $w \le 1,\:\mu$-a.e., the formula given by (\ref{schwartz_solution}), the Cauchy-Schwarz inequality and the Plancherel theorem, we obtain
\begin{eqnarray} \label{main_thm_pf_eq9}
\left\|S^*_{\psi_2}f_1\right\|_{L^{2,\infty}(w\:d\mu)} \le \left\|S^*_{\psi_2}f_1\right\|_{L^2(d\mu)} \lesssim_{\mu} \left\|S^*_{\psi_2}f_1\right\|_{L^\infty(d\mu)} &\lesssim_{k_0}& \|f_1\|_{L^2(\R^n)} \nonumber\\
&\le & \|f\|_{H^{s_0}(\R^n)} \:. 
\end{eqnarray}
Next, for large frequency, for each $x \in \mathbb{B}(0,1)$ and $t \in \I$, once again by the formula (\ref{schwartz_solution}), we see that  
\begin{eqnarray*}
&&\left|e^{it\psi_1(\sqrt{-\Delta})}f_2(x) - e^{it\psi_2(\sqrt{-\Delta})}f_2(x)\right| \\
&=&\left|\frac{1}{(2\pi)^n}\int_{\R^n} e^{ix\cdot \xi} \left(e^{it\psi_1(\|\xi\|)}-e^{it\psi_2(\|\xi\|)}\right)\:\what{f}_2(\xi)\:d\xi\right| \\
&=& \left|\frac{1}{(2\pi)^n}\int_{\R^n} e^{i\left(x\cdot \xi\:+\:t\psi_1(\|\xi\|)\right)} \left(1-e^{it(\psi_2(\|\xi\|)-\psi_1(\|\xi\|))}\right)\:\what{f}_2(\xi)\:d\xi\right| \\
&=& \left|\frac{1}{(2\pi)^n}\int_{\R^n} e^{i\left(x\cdot \xi\:+\:t\psi_1(\|\xi\|)\right)} \left[1- \sum_{j=0}^\infty \frac{\left\{it(\psi_2(\|\xi\|)-\psi_1(\|\xi\|))\right\}^j}{j!} \right] \:\what{f}_2(\xi)\:d\xi\right| \\
&\le& \sum_{j=1}^\infty \frac{|t|^j}{j!} \left|\frac{1}{(2\pi)^n}\int_{\R^n} e^{i\left(x\cdot \xi\:+\:t\psi_1(\|\xi\|)\right)} \left(\psi_2(\|\xi\|)-\psi_1(\|\xi\|)\right)^j\:\what{f}_2(\xi)\:d\xi\right|\:.
\end{eqnarray*}
By (\ref{main_thm_pf_eq6}) and the fact that $f_2 \in \Ss(\R^n)$ with its Fourier transform supported in $\R^n \setminus \mathbb{B}(0,2^{k_0})$, it follows that for each $j \in \N$, there exists  $g_j \in H^{s_0}(\R^n)$ such that
\begin{equation*}
\what{g}_j(\xi)= \left(\psi_2(\|\xi\|)-\psi_1(\|\xi\|)\right)^j\:\what{f}_2(\xi)\:,\:\:\text{ with } \what{g}_j \in L^1(\R^n)\:.
\end{equation*}
In this case, the canonical choice (\ref{canonical_representative}) of the solution with initial data $g_j$ exists at every point and coincides with the formula (\ref{schwartz_solution}), i.e.,
\begin{equation} \label{extra_inequality}
e^{it\psi_1(\sqrt{-\Delta})} g_j(x) = \frac{1}{(2\pi)^n}\int_{\R^n} e^{i\left(x\cdot \xi\:+\:t\psi_1(\|\xi\|)\right)} \left(\psi_2(\|\xi\|)-\psi_1(\|\xi\|)\right)^j\:\what{f}_2(\xi)\:d\xi\:.
\end{equation} 
Then using the well-known fact that the standard quasi-norm on $L^{2,\infty}(w\:d\mu)$ is equivalent to an actual norm, the Monotone Convergence Theorem and (\ref{extra_inequality}), it follows that
\begin{equation*}
\left\|\sup_{t \in \I}\left|e^{it\psi_1(\sqrt{-\Delta})}f_2 - e^{it\psi_2(\sqrt{-\Delta})}f_2\right|\right\|_{L^{2,\infty}(w\:d\mu)} \lesssim \sum_{j=1}^\infty \frac{1}{j!} \left\|S^*_{\psi_1}g_j\right\|_{L^{2,\infty}(w\:d\mu)}\:.
\end{equation*}
At this juncture, we apply the maximal estimate (\ref{general_inequality}) to obtain,
\begin{equation*}
\left\|S^*_{\psi_1}g_j\right\|_{L^{2,\infty}(w\:d\mu)} \lesssim \|g_j\|_{H^{s_0}(\R^n)}\:.
\end{equation*}
Thus we get,
\begin{equation*}
\left\|\sup_{t \in \I}\left|e^{it\psi_1(\sqrt{-\Delta})}f_2 - e^{it\psi_2(\sqrt{-\Delta})}f_2\right|\right\|_{L^{2,\infty}(w\:d\mu)} \lesssim   \sum_{j=1}^\infty \frac{1}{j!} \|g_j\|_{H^{s_0}(\R^n)}\:.
\end{equation*}
Now by the definition of $g_j$ and the hypothesis (\ref{main_thm_pf_eq6}), it follows that
\begin{equation*}
\|g_j\|_{H^{s_0}(\R^n)} \le C^j\:\|f_2\|_{H^{s_0}(\R^n)}\:,
\end{equation*}
and thus yields,
\begin{eqnarray*}
\left\|\sup_{t \in \I}\left|e^{it\psi_1(\sqrt{-\Delta})}f_2 - e^{it\psi_2(\sqrt{-\Delta})}f_2\right|\right\|_{L^{2,\infty}(w\:d\mu)} &\lesssim &  \left(\sum_{j=1}^\infty \frac{C^j}{j!}\right)\:\|f_2\|_{H^{s_0}(\R^n)} \\
&\lesssim& \|f_2\|_{H^{s_0}(\R^n)}\:.
\end{eqnarray*}
Hence by  the last inequality and (\ref{main_thm_pf_eq5}),
\begin{eqnarray} \label{main_thm_pf_eq10}
\left\|S^*_{\psi_2}f_2\right\|_{L^{2,\infty}(w\:d\mu)} &\lesssim&  \left\|\sup_{t \in \I}\left|e^{it\psi_1(\sqrt{-\Delta})}f_2 - e^{it\psi_2(\sqrt{-\Delta})}f_2\right|\right\|_{L^{2,\infty}(w\:d\mu)}+ \left\|S^*_{\psi_1}f_2\right\|_{L^{2,\infty}(w\:d\mu)} \nonumber\\
&\lesssim& \|f_2\|_{H^{s_0}(\R^n)} \nonumber\\
&\le& \|f\|_{H^{s_0}(\R^n)}\:.
\end{eqnarray}
Finally, plugging (\ref{main_thm_pf_eq9}) and (\ref{main_thm_pf_eq10}) in (\ref{main_thm_pf_eq8}), we obtain
\begin{eqnarray*}
\left\|S^*_{\psi_2}f\right\|_{L^{2,\infty}(w\:d\mu)} &\lesssim & \left\|S^*_{\psi_2}f_1\right\|_{L^{2,\infty}(w\:d\mu)}\:+\:\left\|S^*_{\psi_2}f_2\right\|_{L^{2,\infty}(w\:d\mu)} \\
&\lesssim& \|f\|_{H^{s_0}(\R^n)}\:.
\end{eqnarray*}
By density of $\Ss(\R^n)$, the above maximal estimate extends to all of $H^{s_0}(\R^n)$. Thus summarizing, we have proved that for all $s_0>0$ and $\varepsilon>0$,
\begin{equation} \label{main_thm_pf_eq11}
\left\|S^*_{\psi_2}f\right\|_{L^{2,\infty}(w\:d\mu)} \lesssim  \|f\|_{H^{s_0}(\R^n)}\:,\:\:\text{ for all } f \in H^{s_0}(\R^n)\:,
\end{equation}
where $\mu \in \M^\alpha\left(\mathbb{B}(0,1)\right)$ and $\alpha=\fd_{\psi_1,n}(s_0)+\varepsilon$.

\subsection*{Step 3: weighted, weak, fractal maximal estimates to size estimates of divergence sets} Now by the standard method of maximal functions using the maximal estimate (\ref{main_thm_pf_eq11}), density of $\Ss(\R^n)$ and the fractal Sobolev embedding for the initial data (Corollary \ref{fractal-Sobolev-embedding}), we get that for all $\mu \in \M^\alpha\left(\mathbb{B}(0,1)\right)$ and all $f \in H^{s_0}(\R^n)$,
\begin{equation*}
\mu \left\{x \in \mathbb{B}(0,1) : \displaystyle\lim_{t \to 0} e^{it\psi_2(\sqrt{-\Delta})}f(x) \ne f(x)\right\}=0\:.
\end{equation*}
Thus by Frostmann's lemma,
\begin{equation*}
\mathcal{H}^\alpha \left\{x \in \mathbb{B}(0,1) : \displaystyle\lim_{t \to 0} e^{it\psi_2(\sqrt{-\Delta})}f(x) \ne f(x)\right\}=0\:,
\end{equation*}
for all $f \in H^{s_0}(\R^n)$. By translation invariance and countable additivity of Hausdorff measure, we get that
\begin{equation*}
\mathcal{H}^\alpha\left(\D_{\psi_2}(f)\right)=\mathcal{H}^\alpha \left\{x \in \R^n : \displaystyle\lim_{t \to 0} e^{it\psi_2(\sqrt{-\Delta})}f(x) \ne f(x)\right\}=0\:,
\end{equation*}
for all $f \in H^{s_0}(\R^n)$. This implies that 
\begin{equation*} 
\fd_{\psi_2,n}(s_0) \le \alpha = \fd_{\psi_1,n}(s_0)+\varepsilon\:.
\end{equation*}
Since, the above inequality is true for any $\varepsilon>0$, we get that
\begin{equation*} 
\fd_{\psi_2,n}(s_0) \le \fd_{\psi_1,n}(s_0)\:.
\end{equation*}
By the symmetry of the argument, switching the roles of $\psi_1$ and $\psi_2$, we get the opposite inequality and thus getting the result. \qed

\section{Some applications}
\label{sec5}
We now present some applications of Theorem \ref{transference_theorem} for general classes of abstract dispersive equations that encompass several important known examples (see Table \ref{table:1}).

\begingroup

\setlength{\tabcolsep}{10pt} 
\renewcommand{\arraystretch}{1.5} 
\begin{table}[H]
\centering
\begin{tabular}{ ||p{2cm}|| p{4cm} || p{2cm} || p{4cm}|| }
 \hline
 Dispersive equation & Phase of the Fourier multiplier $\psi(\|\xi\|)$ & Dispersive equation & Phase of the Fourier multiplier $\psi(\|\xi\|)$ \\
 \hline 
  Schr\"odinger &  $\|\xi\|^2$   & Fractional Schr\"odinger & $\|\xi\|^a$,\:\:\:$a>0$   \\
 \hline
Linearized Boussinesq  & $\|\xi\| \sqrt{1+\|\xi\|^2}$  & Beam   &  $\sqrt{1+\|\xi\|^4}$ \\
  \hline
Wave & $\|\xi\|$ & Klein-Gordon &  $\sqrt{\|\xi\|^2+a}$,\:\:\:$a > 0$  \\
\hline
\end{tabular}

\medskip

\caption{Some dispersive equations and their phases}
\label{table:1}
\end{table}
\endgroup

The Schr\"odinger equation is the simplest example of the more general family of fractional Schr\"odinger equations:
\begin{equation} \label{frac_schrodinger}
\begin{cases}
	 i\frac{\partial u}{\partial t} +\left(-\Delta \right)^{\frac{a}{2}}u=0\:,\:\:\:  &(x,t) \in \R^n \times \R\:,\:\:a>0\:, \\
	u(\cdot, 0)=f\:, &\text{ on } \R^n \:.
	\end{cases}
\end{equation}
While the classical Schrödinger equation models quantum systems through the lens of continuous Gaussian processes and Brownian motion, in generality, the fractional Schr\"odinger equations provide a quantum mechanical framework for particles experiencing non-Gaussian L\'evy stochastic dynamics. 

Now depending on whether the  phase is convex or concave, i.e., $a>1$ or $a \in (0,1)$ respectively, the analysis of divergence sets differ substantially. We first focus on the {\em `asymptotically convex'} case: 
\begin{corollary} \label{corollary1}
For a dispersive equation with asymptotically convex phase $\psi$ on $\R$ i.e. 
\begin{equation*}
\psi(|\xi|)=|\xi|^a \:+\:\mathcal{O}(1)\:,\:\:|\xi| \gg 1\:,\:\:a>1\:,
\end{equation*}
we have
\begin{equation*} 
\fd_{\psi,1}(s) = \begin{cases}
1\:,\:\:& s<\frac{1}{4} \\
1-2s\:,\:\:& \frac{1}{4} \le s \le \frac{1}{2} \\
0\:,\:\:& \frac{1}{2} \le s\:. 
\end{cases}
\end{equation*}
\end{corollary}

\begin{remark} \label{recovering_results_beam}
\begin{enumerate}
\item Apart from the fractional Schr\"odinger equations with convex phases, the class of dispersive equations in Corollary \ref{corollary1} also contains the known examples: the linearized Boussinesq and the Beam equation, also known as the fourth order wave equation (see Table \ref{table:1}). The Beam equation physically models how flexural stiffness drives structural bending, causing shorter mechanical waves to travel faster than longer ones.    
\item Corollary \ref{corollary1} recovers the results for the linearized Boussinesq \cite[Theorem 1.3]{LiLi} and the Beam propagations \cite[Theorem 1]{LLN} by means of a simple one-line proof, compared to the oscillatory integral estimates involving higher derivatives in \cite{LiLi,LLN}.
\end{enumerate}
\end{remark}

The graph of $\fd_{\psi,1}$ for asymptotically convex phase $\psi$ is presented in red in Figure \ref{figure3}. 

\begin{figure}[htbp]
  \centering
  
  \begin{minipage}{0.48\textwidth}
    \centering
    \begin{tikzpicture}
	\draw [->, thick] (0,0) -- (2,0) node [below left] {$s$};
	\draw [->, thick] (0,0) -- (0,3)  node [below left] {$\fd_{\psi,1}(s)$};
	\draw[red] [-, thick] (0,2) -- (0.5,2);
	\draw[red][thick] plot[domain=0.5:1, smooth] (\x,{2-2*\x});
	\draw[red] [-, thick] (1,0) -- (2,0);
	\draw [fill] (0.5,0) circle [radius=.05] node [below] {\scriptsize$\frac{1}{4}$};
	\draw [fill] (0,2) circle [radius=.05] node [left] {\scriptsize$1$};
	\draw [fill] (0,1) circle [radius=.05] node [left] {\scriptsize$\frac{1}{2}$};
	\draw [fill] (1,0) circle [radius=.05] node [below] {\scriptsize$\frac{1}{2}$};
	\draw[gray] [-, dashed] (0.5,0) -- (0.5,2);
	\draw[gray] [-, dashed] (0,1) -- (0.5,1);
\end{tikzpicture}
    \caption{$\fd_{\psi,1}$ for asymptotically convex phase $\psi$}
    \label{figure3}
  \end{minipage}
  \hfill 
  \begin{minipage}{0.48\textwidth}
    \centering
    \begin{tikzpicture}
	\draw [->, thick] (0,0) -- (2,0) node [below left] {$s$};
	\draw [->, thick] (0,0) -- (0,3)  node [below left] {$\fd_{\psi,1}(s)$};
	\draw[red] [-, thick] (0,2) -- (0.25,2);
	\draw[cyan][thick] plot[domain=0.25:0.5, smooth] (\x,{3-4*\x});
	\draw[green][thick] plot[domain=0.25:0.5, smooth] (\x,{2-2*\x});
	\draw[red][thick] plot[domain=0.5:1, smooth] (\x,{2-2*\x});
	\draw[red] [-, thick] (1,0) -- (2,0);
	\draw [fill] (0.25,0) circle [radius=.05] node [below] {\scriptsize$\frac{a}{4}$};
	\draw [fill] (0.5,0) circle [radius=.05] node [below] {\scriptsize$\frac{1}{4}$};
	\draw [fill] (0,2) circle [radius=.05] node [left] {\scriptsize$1$};
	\draw [fill] (0,1) circle [radius=.05] node [left] {\scriptsize$\frac{1}{2}$};
	\draw [fill] (1,0) circle [radius=.05] node [below] {\scriptsize$\frac{1}{2}$};
	\draw [fill] (0,1.5) circle [radius=.05] node [left] {\scriptsize$1-\frac{a}{2}$};
	\draw[gray] [-, dashed] (0.25,0) -- (0.25,2);
	\draw[gray] [-, dashed] (0,1.5) -- (0.25,1.5);
	\draw[gray] [-, dashed] (0.5,0) -- (0.5,1);
	\draw[red] [-, thick] (1.5,2) -- (1.7,2);
	\pgftext[at={\pgfpoint{3cm}{2cm}}]{\textcolor{red}{sharp result}};
	\draw[cyan] [-, thick] (1.5,1.5) -- (1.7,1.5);
	\draw[gray] [-, dashed] (0,1) -- (0.5,1);
	\pgftext[at={\pgfpoint{3cm}{1.5cm}}]{\textcolor{cyan}{upper bound}};	
	\draw[green] [-, thick] (1.5,1) -- (1.7,1);
	\pgftext[at={\pgfpoint{3cm}{1cm}}]{\textcolor{green}{lower bound}};	
\end{tikzpicture}
    \caption{$\fd_{\psi,1}$ for asymptotically concave phase $\psi$}
    \label{figure4}
  \end{minipage}
\end{figure}

We now present the proof of Corollary \ref{corollary1}:
\begin{proof}[Proof of Corollary \ref{corollary1}]
In the case of the fractional Schr\"odinger equations with convex phase, in dimension one, Sj\"olin proved that $\fd_{\psi,1}(s) =1$ for $s < 1/4$ \cite{Sjolin}. Then Barcel\'o et. al. established that $\fd_{\psi,1}(s) = 1-2s$  for $1/4 \le s \le 1/2$, in \cite{BBCR}. The result for $s>1/2$ follows simply by Sobolev embedding.

The result for dispersive equations with asymptotically convex phase, in Corollary \ref{corollary1}, now follows by Theorem \ref{transference_theorem} and the above mentioned results for the fractional Schr\"odinger equations.
\end{proof}

 Next, we shift to {\em `asymptotically concave phases'}:
\begin{corollary} \label{corollary2}
For a dispersive equation with asymptotically concave phase $\psi$ on $\R$ i.e. 
\begin{equation*}
\psi(|\xi|)=|\xi|^a \:+\:\mathcal{O}(1)\:,\:\:|\xi| \gg 1\:,\:\:a \in (0,1)\:,
\end{equation*}
we have
\begin{enumerate}
\item the sharp result in the range,
\begin{equation*}
\fd_{\psi,1}(s) =\begin{cases}
1\:,\:\: &s<\frac{a}{4}\\
1-2s\:,\:\: &\frac{1}{4} \le s \le \frac{1}{2}\\
0\:,\:\: &\frac{1}{2} \le s\:;
\end{cases}
\end{equation*}
\item and in the remaining range, \\
the upper bound,
\begin{equation*}
\fd_{\psi,1}(s) \le \frac{1}{2}+\frac{1-4s}{2(1-a)}\:,\:\: \frac{a}{4} \le s \le \frac{1}{4}\:,
\end{equation*}
and the lower bound,
\begin{equation*}
\fd_{\psi,1}(s) \ge 1-2s\:,\:\: \frac{a}{4} \le s \le \frac{1}{4}\:;
\end{equation*}
\end{enumerate}
\end{corollary}

The graph of $\fd_{\psi,1}$ for asymptotically concave phase $\psi$ is presented in Figure \ref{figure4}. 

\begin{proof}[Proof of Corollary \ref{corollary2}]
In the case of the fractional Schr\"odinger equations with concave phase i.e. $a \in (0,1)$, in dimension one, Walther proved that $\fd_{\psi,1}(s)=1$ for $s<a/4$ \cite{Walther1}.  Then in the remaining range, Cho-Shiraki recently established that
\begin{equation*} 
\fd_{\psi,1}(s) \le \begin{cases}
 \frac{1}{2}+\frac{1-4s}{2(1-a)}\:,\:\:& \frac{a}{4} \le s \le \frac{1}{4} \\
 1-2s\:,\:\:& \frac{1}{4} \le s \le \frac{1}{2}\:, 
\end{cases}
\end{equation*}
in \cite{CS}. Their bound is sharp in the range $s \in [1/4,1/2]$, by the singularity of Sobolev functions \cite{Zubrinic}, which for the range $s \in [a/4,1/4]$ also yields the lower bound $\fd_{\psi,1}(s) \ge 1-2s$.  The result for $s>1/2$ follows simply by Sobolev embedding.  
 
The result for dispersive equations with asymptotically concave phase, with $a \in (0,1)$, in Corollary \ref{corollary2}, now follows by Theorem \ref{transference_theorem} and the above mentioned results for the fractional Schr\"odinger equations.
\end{proof}

As mentioned in the Introduction,  a complete description of $\fd_{\psi,n}(s)$ for the Schr\"odinger propagation in higher dimensions $n \ge 2$, is currently not known.  By setting the following notation, for $s>0,\:n \in \N$ and $m \in \N \cup \{0\}$,
\begin{eqnarray*}
\alpha^{(1)}_m(s)&:=& n-\frac{2(n-m+1)s-n}{n-m-1}\:, \\
\alpha^{(2)}_m(s)&:=& n+1-\frac{2(n-m+1)s}{n-m}\:,\\
\alpha^{(3)}_m(s)&:=& n - \frac{(n-m-1)(2s-1)}{n-m-2}\:,\\
\mathfrak{s}^{(1)}_m &:=& \frac{(m+1)(n-m)}{2(n-m+1)}\:,\\
\mathfrak{s}^{(2)}_m &:=& \frac{m(n-m)-2m-1}{2(n-m-3)}\:,
\end{eqnarray*}
we generalize the current best known results for the Schr\"odinger propagation to general dispersive equations of {\em `asymptotically quadratic phase'}: 

\begin{corollary} \label{corollary3}
For a dispersive equation with asymptotically quadratic phase $\psi$ on $\R^n,\:n \ge 2$ i.e. 
\begin{equation*}
\psi(\|\xi\|)=\|\xi\|^2 \:+\:\mathcal{O}(1)\:,\:\:\|\xi\| \gg 1\:,
\end{equation*}
we have,
\begin{enumerate}
\item the sharp result for the following range:
\begin{equation*} 
\fd_{\psi,n}(s)  =\begin{cases}
n\:,\:\:& s<\frac{n}{2(n+1)} \\
 n-2s\:,\:\:& \frac{n}{4} \le s \le \frac{n}{2} \\
0\:,\:\:& \frac{n}{2} \le s\:; 
\end{cases}
\end{equation*}
\item in the remaining range, 
\begin{enumerate}
\item the upper bound,
\begin{equation*}
\fd_{\psi,n}(s)  \le n+1-\frac{2(n+1)s}{n}\:,\:\: \frac{n}{2(n+1)} \le s < \frac{n}{4}\:;
\end{equation*}
\item the lower bounds (with $m_0=\lfloor (n-1)/3\rfloor$ and $m_1=\lfloor (n/2)-1\rfloor$),
\begin{enumerate}
\item for dimensions $n=2,3$,
\begin{equation*}
\fd_{\psi,n}(s)  \ge \alpha^{(1)}_0(s)\:,\:\:\frac{n}{2(n+1)} \le s \le \frac{n}{4}\:;
\end{equation*}
\item for dimensions $n=4,5,6,7$,
\begin{equation*}
\fd_{\psi,n}(s)  \ge \begin{cases}
\max\left\{\alpha^{(1)}_{m-1}(s)\:,\:\alpha^{(2)}_m(s)\right\}\:,\:\:& \mathfrak{s}^{(1)}_{m-1} \le s \le \mathfrak{s}^{(1)}_{m}\:,\:\: 1\le m \le m_0  \\
\alpha^{(1)}_{m_0}(s)\:,\:\:& \mathfrak{s}^{(1)}_{m_0} \le s \le \frac{n}{4}\:; 
\end{cases}
\end{equation*}
\item for dimensions $n=8,9,10,11,13$,
\begin{equation*}
\fd_{\psi,n}(s)  \ge \begin{cases}
\max\left\{\alpha^{(1)}_{m-1}(s)\:,\:\alpha^{(2)}_m(s)\right\}\:,\:\:& \mathfrak{s}^{(1)}_{m-1} \le s \le \mathfrak{s}^{(1)}_{m}\:,\:\: 1\le m \le m_0  \\
\max\left\{\alpha^{(1)}_{m_0}(s)\:,\:\alpha^{(3)}_{m_0+1}(s)\right\}\:,\:\:& \mathfrak{s}^{(1)}_{m_0} \le s \le \mathfrak{s}^{(2)}_{m_0+1}  \\
\alpha^{(1)}_{m_0+1}(s)\:,\:\:& \mathfrak{s}^{(2)}_{m_0+1} \le s \le \frac{n}{4}\:; 
\end{cases}
\end{equation*}
\item for dimensions $n\ge 12$, even,
\begin{equation*}
\fd_{\psi,n}(s)  \ge \begin{cases}
\max\left\{\alpha^{(1)}_{m-1}(s)\:,\:\alpha^{(2)}_m(s)\right\}\:,\:\:& \mathfrak{s}^{(1)}_{m-1} \le s \le \mathfrak{s}^{(1)}_{m}\:,\:\: 1\le m \le m_0  \\
\max\left\{\alpha^{(1)}_{m_0}(s)\:,\:\alpha^{(3)}_{m_0+1}(s)\right\}\:,\:\:& \mathfrak{s}^{(1)}_{m_0} \le s \le \mathfrak{s}^{(2)}_{m_0+1}  \\
\max\left\{\alpha^{(1)}_{m-1}(s)\:,\:\alpha^{(3)}_m(s)\right\}\:,\:\:& \mathfrak{s}^{(2)}_{m-1} \le s \le \mathfrak{s}^{(2)}_{m}\:,\:\: m_0+2\le m \le m_1\:; 
\end{cases}
\end{equation*}
\item for dimensions $n\ge 15$, odd,
\begin{equation*}
\fd_{\psi,n}(s)  \ge \begin{cases}
\max\left\{\alpha^{(1)}_{m-1}(s)\:,\:\alpha^{(2)}_m(s)\right\}\:,\:\:& \mathfrak{s}^{(1)}_{m-1} \le s \le \mathfrak{s}^{(1)}_{m}\:,\:\: 1\le m \le m_0  \\
\max\left\{\alpha^{(1)}_{m_0}(s)\:,\:\alpha^{(3)}_{m_0+1}(s)\right\}\:,\:\:& \mathfrak{s}^{(1)}_{m_0} \le s \le \mathfrak{s}^{(2)}_{m_0+1}  \\
\max\left\{\alpha^{(1)}_{m-1}(s)\:,\:\alpha^{(3)}_m(s)\right\}\:,\:\:& \mathfrak{s}^{(2)}_{m-1} \le s \le \mathfrak{s}^{(2)}_{m}\:,\:\: m_0+2\le m \le m_1\\
\alpha^{(1)}_{m_1}(s)\:,\:\:& \mathfrak{s}^{(2)}_{m_1} \le s \le \frac{n}{4}\:.
\end{cases}
\end{equation*}
\end{enumerate}
\end{enumerate} 
\end{enumerate}
\end{corollary}

\begin{remark} \label{new_results_beam}
As special cases, Corollary \ref{corollary3} provides new results for the linearized Boussinesq and the Beam propagations in higher dimensions.
\end{remark}

The graphs of $\fd_{\psi,n}$ for asymptotically quadratic phase $\psi$ are presented in Figure \ref{figure5} for dimensions $n=2,4$.

\begin{figure}[htbp]
  \centering
  
  \begin{subfigure}{0.48\textwidth}
    \centering
    \begin{tikzpicture}
\draw [->, thick] (0,0) -- (4,0) node [below left] {$s$};
\draw [->, thick] (0,0) -- (0,7)  node [below left] {$\fd_{\psi,2}(s)$};
\draw[red] [-, thick] (0,6) -- (1,6);
\draw[red][thick] plot[domain=1.5:3, smooth] (\x,{6-2*\x});
\draw[red] [-, thick] (3,0) -- (4,0);
\draw[cyan][thick]plot[domain=1:1.5, smooth] (\x,{9-3*\x});
\draw[green][thick]plot[domain=1:1.5, smooth] (\x,{12-6*\x});
\draw [fill] (1,0) circle [radius=.05] node [below] {\scriptsize$\frac{1}{3}$};
\draw [fill] (0,3) circle [radius=.05] node [left] {\scriptsize$1$};
\draw [fill] (3,0) circle [radius=.05] node [below] {\scriptsize$1$};
\draw [fill] (0,4.5) circle [radius=.05] node [left] {\scriptsize$\frac{3}{2}$};
\draw [fill] (0,6) circle [radius=.05] node [left] {\scriptsize$2$};
\draw [fill] (1.5,0) circle [radius=.05] node [below] {\scriptsize$\frac{1}{2}$};
\draw[gray] [-, dashed] (1,0) -- (1,6);
\draw[gray] [-, dashed] (1.5,0) -- (1.5,4.5);
\draw[gray] [-, dashed] (0,3) -- (1.5,3);
\draw[gray] [-, dashed] (0,4.5) -- (1.5,4.5);
\draw[red] [-, thick] (3,6) -- (3.2,6);
\pgftext[at={\pgfpoint{4.5cm}{6cm}}]{\textcolor{red}{sharp result}};
\draw[cyan] [-, thick] (3,5.5) -- (3.2,5.5);
\pgftext[at={\pgfpoint{4.6cm}{5.5cm}}]{\textcolor{cyan}{upper bound}};	
\draw[green] [-, thick] (3,5) -- (3.2,5);
\pgftext[at={\pgfpoint{4.5cm}{5cm}}]{\textcolor{green}{lower bound}};
\end{tikzpicture}
    \caption{dimension $2$}
    \label{figure5.1}
  \end{subfigure}
  \hfill 
  \begin{subfigure}{0.48\textwidth}
    \centering
    \begin{tikzpicture}
\draw [->, thick] (0,0) -- (5,0) node [below left] {$s$};
\draw [->, thick] (0,0) -- (0,9)  node [below left] {$\fd_{\psi,4}(s)$};
\draw[red] [-, thick] (0,8) -- (0.8,8);
\draw[red][thick] plot[domain=2:4, smooth] (\x,{8-2*\x});
\draw[red] [-, thick] (4,0) -- (5,0);
\draw[cyan][thick]plot[domain=0.8:2, smooth] (\x,{10-2.5*\x});
\draw[green][thick]plot[domain=0.8:1, smooth] (\x,{10.666-3.333*\x});
\draw[green][thick]plot[domain=1:1.5, smooth] (\x,{10-2.666*\x});
\draw[green][thick]plot[domain=1.5:2, smooth] (\x,{12-4*\x});
\draw [fill] (0.8,0) circle [radius=.05] node [below] {\scriptsize$\frac{2}{5}$};
\draw[gray] [-, dashed] (0.8,0) -- (0.8,8);
\draw [fill] (4,0) circle [radius=.05] node [below] {\scriptsize$2$};
\draw [fill] (0,4) circle [radius=.05] node [left] {\scriptsize$2$};
\draw [fill] (2,0) circle [radius=.05] node [below] {\scriptsize$1$};
\draw [fill] (0,5) circle [radius=.05] node [left] {\scriptsize$\frac{5}{2}$};
\draw[gray] [-, dashed] (2,0) -- (2,5);
\draw[gray] [-, dashed] (0,5) -- (2,5);
\draw[gray] [-, dashed] (0,4) -- (2,4);
\draw[red] [-, thick] (3,6) -- (3.2,6);
\pgftext[at={\pgfpoint{4.5cm}{6cm}}]{\textcolor{red}{sharp result}};
\draw[cyan] [-, thick] (3,5.5) -- (3.2,5.5);
\pgftext[at={\pgfpoint{4.6cm}{5.5cm}}]{\textcolor{cyan}{upper bound}};	
\draw[green] [-, thick] (3,5) -- (3.2,5);
\pgftext[at={\pgfpoint{4.5cm}{5cm}}]{\textcolor{green}{lower bound}};
\end{tikzpicture}
    \caption{dimension $4$}
    \label{figure5.2}
  \end{subfigure}
  \caption{$\fd_{\psi,n}$ for asymptotically quadratic phase $\psi$}
  \label{figure5}
\end{figure}
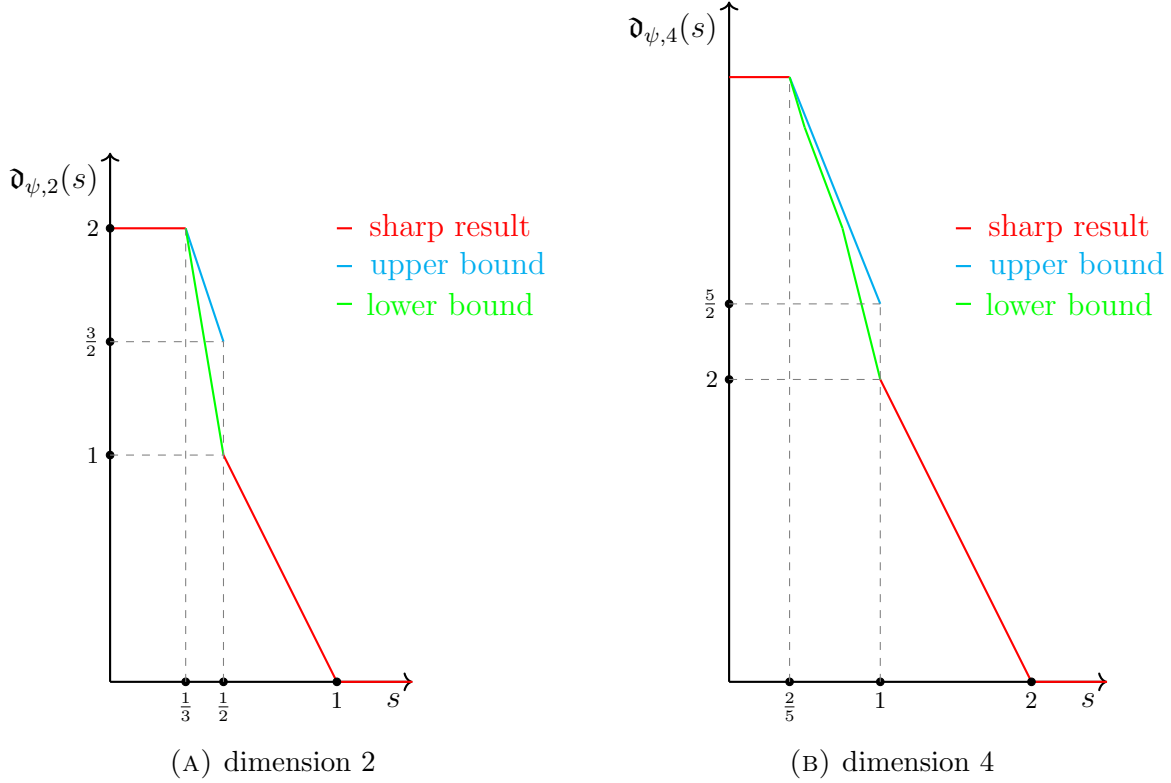

We now present the proof of Corollary \ref{corollary3}:
\begin{proof}[Proof of Corollary \ref{corollary3}]
For the Schr\"odinger propagation in dimensions $n \ge 2$,
\begin{itemize}
\item as mentioned in the Introduction, the sharp results are known in the range $s \in (0,n/2(n+1)) \cup [n/4,\infty)$ and are due to Bourgain \cite{Bourgain}, Barcel\'o et. al. \cite{BBCR} and the Sobolev embedding;
\item in the remaining range, the upper bounds and lower bounds mentioned in part (2) of Corollary \ref{corollary3} are due to Du et. al. \cite{DGL,DZ} and Eceizabarrena et. al. \cite{EP} respectively.
\end{itemize}

The result for dispersive equations with asymptotically quadratic phase, in Corollary \ref{corollary3}, now follows by Theorem \ref{transference_theorem} and the above mentioned results for the Schr\"odinger equation.
\end{proof}

We now focus on the case complementary to both convex ($a>1$) and concave ($a \in (0,1)$) phases, namely the Wave operator or the half-Schr\"odinger, corresponding to $a=1$ in (\ref{frac_schrodinger}). It is related to the Cauchy problem:
\begin{equation*}
\frac{\partial^2 u}{\partial t^2} - \Delta u=0,\:\:u(\cdot,0)=u_0,\:\: \frac{\partial u}{\partial t}(\cdot, 0)=u_1\:,
\end{equation*}
where $(u_0,u_1) \in H^s(\R^n) \times H^{s-1}(\R^n)$, and represents the dynamic mechanism by which energy, information, and physical disturbances travel through space and time without the permanent transport of matter.

For `wave-like' dispersive equations, we have
\begin{corollary} \label{corollary4}
For a dispersive equation with asymptotically linear phase $\psi$ i.e. 
\begin{equation*}
\psi(\|\xi\|)=\|\xi\| \:+\:\mathcal{O}(1)\:,\:\:\|\xi\| \gg 1\:,
\end{equation*}
we have,
\begin{enumerate}
\item for $n \ge 1$, the sharp result for the following range:
\begin{equation*}
\fd_{\psi,n}(s) = \begin{cases}
n, \:\:\:&s \le \frac{1}{2}\\
0, \:\:\:&s>\frac{n}{2}\:;
\end{cases}
\end{equation*}
\item in the remaining range, 
\begin{enumerate}
\item for dimensions $n=2,3$, the sharp results:
\begin{equation*}
\fd_{\psi,2}(s) = \begin{cases}
4-4s, \:\:\:&\frac{1}{2} \le s <\frac{3}{4}\\
2-2s, \:\:\:&\frac{3}{4}< s \le 1\:;
\end{cases}
\text{ and } \fd_{\psi,3}(s) = \begin{cases}
5-4s, \:\:\:&\frac{1}{2} \le s \le 1\\
3-2s, \:\:\:&1 \le s \le \frac{3}{2}\:;
\end{cases}
\end{equation*}
\item for dimensions $n \ge 4$, 
\begin{enumerate}
\item  the sharp result in the range,
\begin{equation*}
\fd_{\psi,n}(s) = n-2s, \:\:\:\frac{n+1}{4} \le s \le \frac{n}{2}\:;
\end{equation*}
\item the upper bound,
\begin{equation*}
\fd_{\psi,n}(s) \le \begin{cases}
\frac{n^2-n-1}{n-2}-\frac{2(n-1)s}{n-2}, \:\:\:&\frac{1}{2} < s \le \frac{n}{4}\\
\frac{3n+1}{2}-4s, \:\:\:&\frac{n}{4} \le s \le \frac{n+1}{4}\:;
\end{cases}
\end{equation*}
\item the lower bound,
\begin{equation*}
\fd_{\psi,n}(s) \ge \begin{cases}
n+2-4s, \:\:\:&\frac{1}{2} < s \le 1\\
n-2s, \:\:\:&1 \le s \le \frac{n+1}{4}\:.
\end{cases}
\end{equation*}
\end{enumerate}
\end{enumerate}
\end{enumerate}
\end{corollary}

\begin{remark} \label{new_results_klein}
As special cases, Corollary \ref{corollary4} provides new results for the family of Klein-Gordon propagations (see Table \ref{table:1}), in particular, solving the classical Carleson's problem for Lebesgue measure. These equations are related to relativistic wave equations that describe how mass affects wave propagation in high-energy physics.
\end{remark}

The graphs of $\fd_{\psi,n}$ for asymptotically linear phase $\psi$ are presented in Figure \ref{figure6} for dimensions $n=2,3,4$.

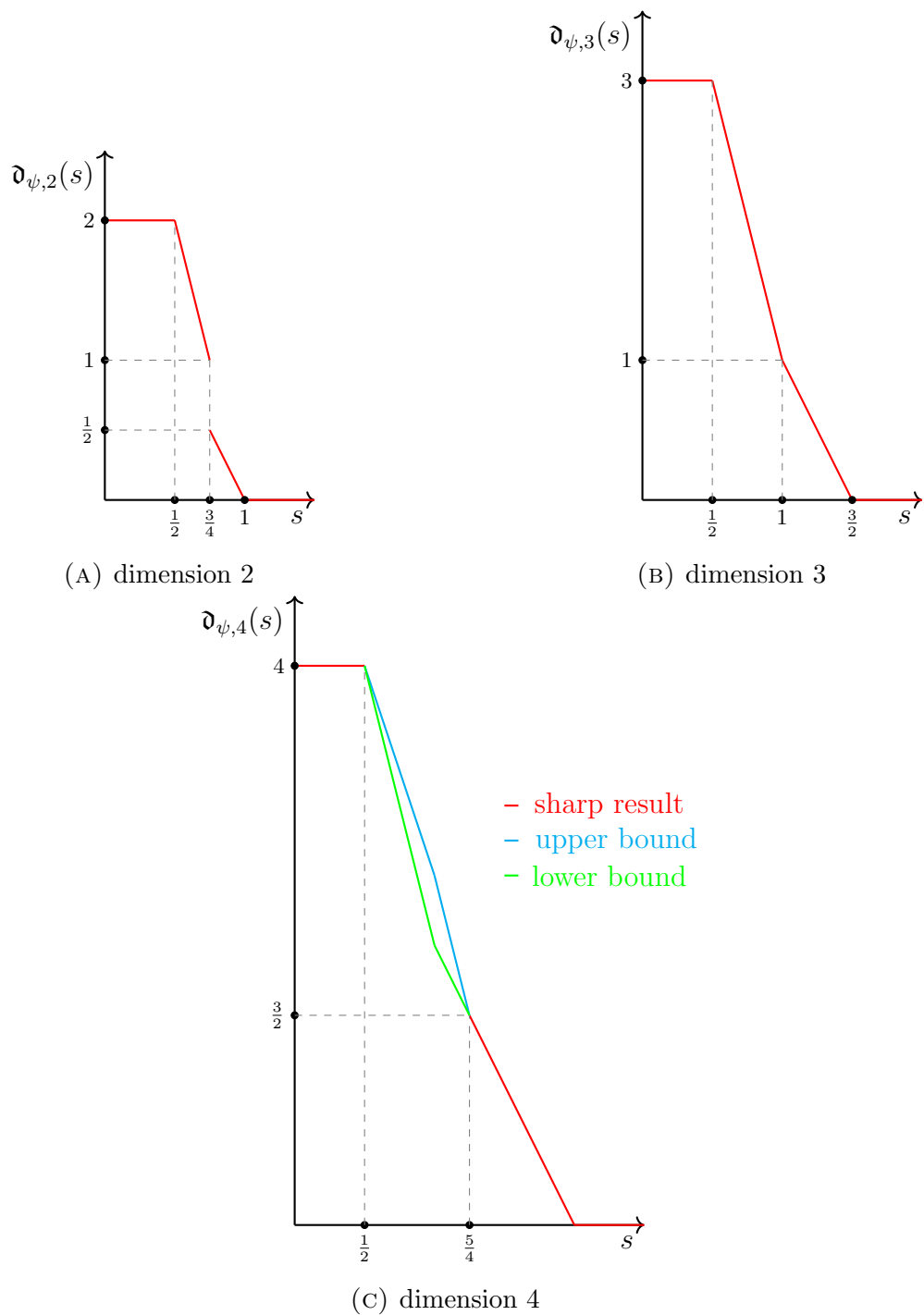
\begin{figure}[htbp]
  \centering
  
  \begin{subfigure}{0.48\textwidth}
    \centering
    \begin{tikzpicture}
\draw [->, thick] (0,0) -- (3,0) node [below left] {$s$};
\draw [->, thick] (0,0) -- (0,5)  node [below left] {$\fd_{\psi,2}(s)$};
\draw[red] [-, thick] (0,4) -- (1,4);
\draw[red][thick] plot[domain=1:1.5, smooth] (\x,{8-4*\x});
\draw[red][thick] plot[domain=1.5:2, smooth] (\x,{4-2*\x});
\draw[red] [-, thick] (2,0) -- (3,0);
\draw [fill] (0,4) circle [radius=.05] node [left] {\scriptsize$2$};
\draw [fill] (0,2) circle [radius=.05] node [left] {\scriptsize$1$};
\draw [fill] (0,1) circle [radius=.05] node [left] {\scriptsize$\frac{1}{2}$};
\draw [fill] (1,0) circle [radius=.05] node [below] {\scriptsize$\frac{1}{2}$};
\draw [fill] (1.5,0) circle [radius=.05] node [below] {\scriptsize$\frac{3}{4}$};
\draw [fill] (2,0) circle [radius=.05] node [below] {\scriptsize$1$};
\draw[gray] [-, dashed] (1.5,0) -- (1.5,2);
\draw[gray] [-, dashed] (1,0) -- (1,4);
\draw[gray] [-, dashed] (0,2) -- (1.5,2);
\draw[gray] [-, dashed] (0,1) -- (1.5,1);
\end{tikzpicture}
    \caption{dimension $2$}
    \label{figure6.1}
  \end{subfigure}
  \hfill 
  \begin{subfigure}{0.48\textwidth}
    \centering
    \begin{tikzpicture}
\draw [->, thick] (0,0) -- (4,0) node [below left] {$s$};
\draw [->, thick] (0,0) -- (0,7)  node [below left] {$\fd_{\psi,3}(s)$};
\draw[red] [-, thick] (0,6) -- (1,6);
\draw[red][thick] plot[domain=1:2, smooth] (\x,{10-4*\x});
\draw[red][thick] plot[domain=2:3, smooth] (\x,{6-2*\x});
\draw[red] [-, thick] (3,0) -- (4,0);
\draw [fill] (0,6) circle [radius=.05] node [left] {\scriptsize$3$};
\draw [fill] (0,2) circle [radius=.05] node [left] {\scriptsize$1$};
\draw [fill] (1,0) circle [radius=.05] node [below] {\scriptsize$\frac{1}{2}$};
\draw [fill] (2,0) circle [radius=.05] node [below] {\scriptsize$1$};
\draw [fill] (3,0) circle [radius=.05] node [below] {\scriptsize$\frac{3}{2}$};
\draw[gray] [-, dashed] (1,0) -- (1,6);
\draw[gray] [-, dashed] (2,0) -- (2,2);
\draw[gray] [-, dashed] (0,2) -- (2,2);
\end{tikzpicture}
    \caption{dimension $3$}
    \label{figure6.2}
  \end{subfigure}
  \hfill
  \begin{subfigure}{0.48\textwidth}
    \centering
   \begin{tikzpicture}
\draw [->, thick] (0,0) -- (5,0) node [below left] {$s$};
\draw [->, thick] (0,0) -- (0,9)  node [below left] {$\fd_{\psi,4}(s)$};
\draw[red] [-, thick] (0,8) -- (1,8);
\draw[red][thick] plot[domain=2.5:4, smooth] (\x,{8-2*\x});
\draw[red] [-, thick] (4,0) -- (5,0);
\draw[cyan][thick]plot[domain=1:2, smooth] (\x,{11-3*\x});
\draw[cyan][thick]plot[domain=2:2.5, smooth] (\x,{13-4*\x});
\draw[green][thick]plot[domain=1:2, smooth] (\x,{12-4*\x});
\draw[green][thick]plot[domain=2:2.5, smooth] (\x,{8-2*\x});
\draw [fill] (1,0) circle [radius=.05] node [below] {\scriptsize$\frac{1}{2}$};
\draw [fill] (0,8) circle [radius=.05] node [left] {\scriptsize$4$};
\draw [fill] (2.5,0) circle [radius=.05] node [below] {\scriptsize$\frac{5}{4}$};
\draw [fill] (0,3) circle [radius=.05] node [left] {\scriptsize$\frac{3}{2}$};
\draw[gray] [-, dashed] (1,0) -- (1,8);
\draw[gray] [-, dashed] (2.5,0) -- (2.5,3);
\draw[gray] [-, dashed] (0,3) -- (2.5,3);
\draw[red] [-, thick] (3,6) -- (3.2,6);
\pgftext[at={\pgfpoint{4.5cm}{6cm}}]{\textcolor{red}{sharp result}};
\draw[cyan] [-, thick] (3,5.5) -- (3.2,5.5);
\pgftext[at={\pgfpoint{4.6cm}{5.5cm}}]{\textcolor{cyan}{upper bound}};	
\draw[green] [-, thick] (3,5) -- (3.2,5);
\pgftext[at={\pgfpoint{4.5cm}{5cm}}]{\textcolor{green}{lower bound}};
\end{tikzpicture}
    \caption{dimension $4$}
    \label{figure6.3}
  \end{subfigure}
  \caption{$\fd_{\psi,n}$ for asymptotically linear phase $\psi$}
  \label{figure6}
\end{figure}

\begin{proof}[Proof of Corollary \ref{corollary4}]
For the Wave propagation, by Walther's negative result \cite{Walther2} and the Sobolev embedding, we have for general $n \ge 1$,
\begin{equation*}
\fd_{\psi,n}(s) = \begin{cases}
n, \:\:\:&s\le\frac{1}{2}\\
0, \:\:\:&s>\frac{n}{2}\:.
\end{cases}
\end{equation*}
In dimension $n=2$, Barcel\'o et. al. \cite{BBCR} established the sharp result:
\begin{equation*}
\fd_{\psi,2}(s) = \begin{cases}
4-4s, \:\:\:&\frac{1}{2} \le s <\frac{3}{4}\\
2-2s, \:\:\:&\frac{3}{4}< s \le 1\:.
\end{cases}
\end{equation*}
In dimension $n=3$, Ham-Ko-Lee \cite{Ham} established the sharp result:
\begin{equation*}
\fd_{\psi,3}(s) = \begin{cases}
5-4s, \:\:\:&\frac{1}{2} \le s \le 1\\
3-2s, \:\:\:&1 \le s \le \frac{3}{2}\:.
\end{cases}
\end{equation*}
Finally, for dimension $n \ge 4$, the results stated in (2.b), in the statement of Corollary \ref{corollary4} are due to upper bounds in \cite{Ham} and lower bounds in \cite{BBCR}.   

The result for dispersive equations with asymptotically linear phase, in Corollary \ref{corollary4}, now follows by Theorem \ref{transference_theorem} and the above mentioned results for the Wave propagation.
\end{proof}

\section{Effect of radial symmetry}
\label{sec6}
In this section, we study divergence sets for dispersive equations  on $\R^n$, $n \ge 2$, for radial initial data. In Subsection \ref{subsec6.1}, we start off by building the set-up with some general results. In Subsection \ref{subsec6.2}, we prove Theorem \ref{radial_concave_thm}. Then in Subsection \ref{subsec6.3}, we go on to prove results for general classes of dispersive PDEs, as applications of the transference principle.

\subsection{Some general results} \label{subsec6.1}
We recall the canonical choice of representation of the solution and the initial data given by (\ref{canonical_representative}), which also coincides up to a set of Hausdorff dimension $\le n-2s$, with the classical choice given as a Bessel potential. Then working out the proof of Lemma \ref{N-maximal-estimate} for arbitrary balls $\mathbb{B}(x,r)$ by following the arguments of Appendices A and C in \cite{BBCR}, we have the following result on annuli:
\begin{lemma} \label{N-maximal-estimate-radial}
Let $s>0$ and $\alpha >n-2s$. Then for any $f \in H^s(\R^n)$ and any $\mu \in \M^\alpha(\mathbb{A}(r_1,r_2))$, for $0<r_1<r_2<\infty$ and any $t \in \R$ , we have
\begin{equation*}
\left\|\sup_{N \ge 1} \left|S^N_{\psi,t}f\right|\right\|_{L^2(d\mu)} \lesssim \sqrt{c_\alpha(\mu)}\:\|f\|_{H^s(\R^n)}\:,
\end{equation*}
where the implicit constant depends only on $r_1,r_2$\:.
\end{lemma}

As a consequence of Lemma \ref{N-maximal-estimate-radial}, we have the following  fixed-time fractal Sobolev embedding on annuli:
\begin{lemma} \label{Annulus-Sobolev-embedding}
Let $s>0$ and $\alpha >n-2s$. Then for any $f \in H^s(\R^n)$, any $\mu \in \M^\alpha(\mathbb{A}(r_1,r_2))$, for $0<r_1<r_2<\infty$ and any $t \in \R$, we have
\begin{equation*}
\left\|u(\cdot,t)\right\|_{L^2(d\mu)} \lesssim \sqrt{c_\alpha(\mu)}\:\|f\|_{H^s(\R^n)}\:,
\end{equation*}
where the implicit constant depends only  on $r_1,r_2$\:.
\end{lemma}
\begin{remark} \label{BR_remark}
The index $n-2s$ showing up above, is actually the critical index for Hausdorff dimension of singular sets of radial Sobolev functions $f \in H^s(\R^n)$ for $s \in (0,1/2)$. In \cite{BR}, the authors constructed radial functions $f \in H^s(\R^n)$ for $s \in (0,1/2)$, having singularities on sets of Hausdorff dimension $<n-2s$. Consequently, our problem of estimating Hausdorff dimension of divergence sets becomes ill-posed for $\wtilde{\fd}_{\psi,n}(s)<n-2s$ in the range $s \in (0,1/2)$.
\end{remark} 

In the regularity regime $s>1/2$, however, the situation is drastically better:
\begin{lemma} \label{lemma_s>1/2}
We have for all phases $\psi$ and $s>1/2$,
\begin{equation*}
\left\|\sup_{t \in \I} \sup_{N \ge 1} \left|S^N_{\psi,t}f\right|\right\|_{L^1(d\mu)} \lesssim \|f\|_{H^s(\R^n)}\:,
\end{equation*}
where $\mu \in \M^\alpha(\mathbb{A}(r_1,r_2))$, for $0<r_1<r_2<\infty,\:\alpha>0,\:f \in H^s(\R^n)$ radial and $S^N_{\psi,t}f$ is as in (\ref{N-scale}). The implicit constant in the above inequality depends only on $\mu,\:r_1,\:r_2$\:.
\end{lemma}

\begin{proof}
Let $G_s$ be the Bessel kernel on $\R^n$. Then writing $f=G_s*g$, it suffices to prove that 
\begin{equation} \label{lemma_pf_eq1}
\int_{\mathbb{A}(r_1,r_2)} \sup_{t \in \I} \sup_{N \ge 1} \left|\int_{\R^n} \frac{\Phi\left(N^{-1}\xi\right)\:e^{i\left(x\cdot \xi + t\psi(\|\xi\|)\right)}\:\what{g}(\xi)}{(1+\|\xi\|^2)^{s/2}}\:d\xi\right|\:d\mu(x) \lesssim \|g\|_{L^2(\R^n)}\:,\:\:\text{ if } s>\frac{1}{2}\:,
\end{equation} 
uniformly in $g \in L^2(\R^n)$ radial. As $\mu$ is a Radon measure on $\mathbb{A}(r_1,r_2)$, to prove (\ref{lemma_pf_eq1}), it is enough to prove the pointwise inequality,
\begin{equation} \label{lemma_pf_eq2}
\left|\int_{\R^n} \frac{\Phi\left(N^{-1}\xi\right)\:e^{i\left(x\cdot \xi + t\psi(\|\xi\|)\right)}\:\what{g}(\xi)}{(1+\|\xi\|^2)^{s/2}}\:d\xi\right| \lesssim \|g\|_{L^2(\R^n)}\:,\:\:\text{ if } s>\frac{1}{2}\:,
\end{equation}
uniformly in $N \ge 1,\:t \in \I,\:x \in \mathbb{A}(r_1,r_2)$ and $g \in L^2(\R^n)$ radial. Now using the radial symmetry of the phase $\psi$ and the function $g$, by polar decomposition, we have by writing $\lambda=\|\xi\|$,
\begin{equation*}
\left|\int_{\R^n} \frac{\Phi\left(N^{-1}\xi\right)\:e^{i\left(x\cdot \xi + t\psi(\|\xi\|)\right)}\:\what{g}(\xi)}{(1+\|\xi\|^2)^{s/2}}\:d\xi\right| = C \left|\int_{0}^\infty \frac{\Phi\left(N^{-1}\lambda\right)\lambda^{n-1}\:e^{it\psi(\lambda)}\:\:\what{g}(\lambda)}{(1+\lambda^2)^{s/2}}\frac{J_{\frac{n-2}{2}}(\lambda\|x\|)}{(\lambda\|x\|)^{\frac{n-2}{2}}}\:d\lambda\right|\:,
\end{equation*}
where $J_{(n-2)/2}$ denotes the Bessel function of the first kind of order $(n-2)/2$ and $C$ is some positive constant depending only on $n$. Then using the boundedness of Bessel functions of the first kind near $0$ and decay away from $0$ (see \cite{Szego}), we have the well-known estimate,
\begin{equation*}
\left|(\lambda\|x\|)^{\frac{1}{2}}\:J_{\frac{n-2}{2}}(\lambda\|x\|)\right| \lesssim 1\:,
\end{equation*}
with the implicit constant depending only on $n$. Thus by the Cauchy-Schwarz inequality and the Plancherel theorem, we get uniformly in $N \ge 1,\:t \in \I,\:x \in \mathbb{A}(r_1,r_2)$ and $g \in L^2(\R^n)$ radial,
\begin{eqnarray*}
&&\left|\int_{\R^n} \frac{\Phi\left(N^{-1}\xi\right)\:e^{i\left(x\cdot \xi + t\psi(\|\xi\|)\right)}\:\what{g}(\xi)}{(1+\|\xi\|^2)^{s/2}}\:d\xi\right| \\
&\lesssim & \frac{1}{\|x\|^{\frac{n-1}{2}}} \int_0^\infty \frac{\lambda^{\frac{n-1}{2}} \:|\what{g}(\lambda)|}{(1+\lambda^2)^{s/2}}\:d\lambda\\
& \lesssim & \frac{1}{\|x\|^{\frac{n-1}{2}}} \left(\int_0^\infty \lambda^{n-1}\:|\what{g}(\lambda)|^2\:d\lambda\right)^\frac{1}{2} \left(\int_0^\infty \frac{d\lambda}{(1+\lambda^2)^s}\right)^\frac{1}{2} \\
& \lesssim & \|g\|_{L^2(\R^n)}\:,
\end{eqnarray*}
as $x \in \mathbb{A}(r_1,r_2)$ and $s>1/2$. Hence we obtain (\ref{lemma_pf_eq2}) and this completes the proof of Lemma \ref{lemma_s>1/2}.
\end{proof}

\begin{remark}
The proof of Lemma \ref{lemma_s>1/2} follows from the pointwise decay estimates of Bessel functions and does not involve the oscillation afforded by the Fourier multiplier.
\end{remark}

\subsection{Fractional Schr\"odinger equations with concave phase and the wave propagation} \label{subsec6.2}
In this subsection, we prove Theorem \ref{radial_concave_thm}, but first we recall the following crucial lemma which follows from the proof of \cite[Lemma Appendix B.2.]{Koruku}:
\begin{lemma} \label{CS_lemma}
Let $\alpha \in (0,1]$ and $\nu$ be an $\alpha$-dimensional measure on $\R$. There exists a constant $C>0$, such that for any interval $[c,d] \subset \R$, we have
\begin{equation*}
\left|\int_{\R} \int_{\R} g_1(x)\:g_2(x')\:\chi_{[c,d]}(x-x')\:d\nu(x)\:d\nu(x')\right| \le C(d-c)^{\alpha}\:\|g_1\|_{L^2(d\nu)}\:\|g_2\|_{L^2(d\nu)}\:.
\end{equation*}
Moreover, if $\rho \in (0,\alpha)$, then  there exists a constant $C>0$ such that
\begin{equation*}
\left|\int_{\R} \int_{\R} g_1(x)\:g_2(x')\:|x-x'|^{-\rho}\:d\nu(x)\:d\nu(x')\right| \le C\:\|g_1\|_{L^2(d\nu)}\:\|g_2\|_{L^2(d\nu)}\:.
\end{equation*}
\end{lemma}
We now present the proof of Theorem \ref{radial_concave_thm}:
\begin{proof}[Proof of Theorem \ref{radial_concave_thm}]
We first prove part (1). The result for the range $s \in (0,a/4)$ follows from Walther's work in \cite{Walther3}. We now focus on the range $s \in (a/4,1/2)$ and prove the maximal estimate,
\begin{equation} \label{concave_pf_eq1}
\|S^*_\psi f\|_{L^1(d\mu)} \lesssim \|f\|_{H^s(\R^n)}\:,
\end{equation} 
for $\mu \in \M^\alpha(\mathbb{A}(r_1,r_2))\:,\: 0<r_1<r_2<\infty$, all $f \in \Ss(\R^n)$ radial and 
\begin{equation*}
\alpha > \max \left\{n-\frac{1}{2}\:+\:\frac{1-4s}{2(1-a)}\:,\: n-2s\right\}\:,
\end{equation*}
or equivalently for 
\begin{equation*}
s>\max\left\{\frac{1}{2}-\frac{a}{4}-\frac{(1-a)(\alpha-n+1)}{2}\:,\: \frac{1}{2}-\frac{(\alpha-n+1)}{2}\right\}\:,
\end{equation*}
where $\alpha \in (n-1,n)$\:. Then by considering the Bessel kernel $G_s$, it suffices to prove that
\begin{equation} \label{concave_pf_eq2}
\|S^*_\psi (G_s *f)\|_{L^1(d\mu)} \lesssim \|f\|_{L^2(\R^n)}\:,
\end{equation}
for $\mu \in \M^\alpha(\mathbb{A}(r_1,r_2))\:,\: 0<r_1<r_2<\infty$, all $f \in \Ss(\R^n)$ radial. Denoting $\|x\|$ by $r$, using the radial symmetry of the phase $\psi$ and the function $f$, by polar decomposition, we have 
\begin{eqnarray*}
S_{\psi,t} \left(G_s*f\right)(x)&=&S_{\psi,t} \left(G_s*f\right)(r) \\
&=& \frac{C}{r^{\frac{n-1}{2}}} \int_{0}^\infty \what{f}(\lambda) \:e^{it\lambda^a} (\lambda r)^{\frac{1}{2}} J_{\frac{n-2}{2}}(\lambda r)\frac{\lambda^{\frac{n-1}{2}} }{(1+\lambda^2)^{s/2}}\:d\lambda\:.
\end{eqnarray*} 
Thus it suffices to work with radial measures. We then consider, $d\mu_0(r):=d\mu(x)$, the one-dimensional radial pushforward of $\mu$. Now as for each interval $I_v$ of length $v$, we can find an $n$-ball $B_{nv}$ of radius $nv$ so that,
\begin{equation*}
\mu_0\left(I_v\right)\:v^{n-1} \lesssim \mu\left(B_{nv}\right)\:,
\end{equation*}
it follows that $\mu_0 \in \M^{\alpha-(n-1)}(r_1,r_2)$\:. Hence to prove (\ref{concave_pf_eq2}), it suffices to prove the following estimate:
\begin{equation*} 
\int_{r_1}^{r_2} \sup_{t \in \I} \left|S_{\psi,t} \left(G_s*f\right)(r)\right|\:r^{\frac{n-1}{2}}\:d\mu_0(r) \lesssim \|f\|_{L^2(\R^n)}\:,
\end{equation*}
which in turn by linearization, is equivalent to prove 
\begin{equation} \label{concave_pf_eq3}
\int_{r_1}^{r_2} \left|S_{\psi,t(r)} \left(G_s*f\right)(r)\right|\:r^{\frac{n-1}{2}}\:d\mu_0(r) \lesssim \|f\|_{L^2(\R^n)}\:,
\end{equation}
for $\mu_0 \in \M^{\alpha_0}(r_1,r_2)\:,\: 0<r_1<r_2<\infty$, all $f \in \Ss(\R^n)$ radial and 
\begin{equation*}
s>\max\left\{\frac{1}{2}-\frac{a}{4}-\frac{(1-a)\alpha_0}{2}\:,\: \frac{1}{2}-\frac{\alpha_0}{2}\right\}\:,
\end{equation*}
where $\alpha_0 \in (0,1)$ and $t : (r_1,r_2) \to \I$ is a measurable function. Finally, by duality, to prove (\ref{concave_pf_eq3}), it suffices to prove that
\begin{equation} \label{concave_pf_eq4}
\|Tg\|_{L^2(0,\infty)} \lesssim \|g\|_{L^\infty(d\mu_0)}\:,
\end{equation}
for $g \in L^\infty(d\mu_0)$, where 
\begin{equation*}
Tg(\lambda):= \frac{1}{(1+\lambda^2)^{s/2}} \int_{r_1}^{r_2} (\lambda r)^{\frac{1}{2}} J_{\frac{n-2}{2}}(\lambda r)\:e^{it(r)\lambda^a}\:g(r)\:d\mu_0(r)\:.
\end{equation*}
Now invoking the facts (see \cite{Szego}) that there are complex numbers $a_1,a_2$ and some constant $C>0$, such that 
\begin{eqnarray*}
\left|y^{\frac{1}{2}}\:J_{\frac{n-2}{2}}(y) -\left(a_1e^{iy}+a_2e^{-iy}\right)\right| &\le& Cy^{-1}\:,\:y>1\:,\\
\left|y^{\frac{1}{2}}\:J_{\frac{n-2}{2}}(y) -\left(a_1e^{iy}+a_2e^{-iy}\right)\right| &\le& C\:,\:0<y \le 1\:,
\end{eqnarray*}
we decompose,
\begin{equation*}
Tg(\lambda)=a_1T_1g(\lambda)\:+\:a_2T_1g(-\lambda)\:+\:T_2g(\lambda)\:,
\end{equation*}
where
\begin{equation*}
T_1g(\lambda) := \frac{1}{(1+\lambda^2)^{s/2}} \int_{r_1}^{r_2} e^{i\left(\lambda r+t(r)|\lambda|^a\right)}\:g(r)\:d\mu_0(r)\:,
\end{equation*}
and $T_2g$ is the part corresponding to the remainder term. By \cite[pp. 7-8]{BR}, we have
\begin{equation} \label{concave_pf_eq5}
\|T_2g\|_{L^2(0,\infty)} \lesssim \|g\|_{L^\infty(d\mu_0)}\:.
\end{equation}
We now focus on $T_1$. We set,
\begin{equation*}
Ug(\lambda):=\frac{1}{(1+\lambda^2)^{s/2}} \int_{r_1}^{r_2} e^{i\left(\lambda r+t(r)|\lambda|^a\right)}\:g(r)\:d\mu_0(r)\:,\:\:\lambda \in \R\:,
\end{equation*}
and shall prove that
\begin{equation} \label{concave_pf_eq6}
\|Ug\|_{L^2(\R)} \lesssim \|g\|_{L^\infty(d\mu_0)}\:,
\end{equation}
from which the required estimates for $T_1$ would follow.

We now consider an auxiliary non-negative, even function $\eta \in C^\infty_c(\R)$ such that 
\begin{equation*}
Supp(\eta) \subset \left\{\lambda \in \R : \frac{1}{2} < |\lambda| < 2\right\}\:\:\text{ and } \sum_{k=-\infty}^\infty \eta\left(2^{-k}\lambda\right)=1\:,\:\text{ for } \lambda \ne 0\:.
\end{equation*}
Then setting,
\begin{equation*}
\eta_1(\lambda):=\sum_{k=-\infty}^0 \eta\left(2^{-k}\lambda\right)\:,\:\:\text{ and } \eta_2(\lambda):=\sum_{k=1}^\infty \eta\left(2^{-k}\lambda\right)\:,
\end{equation*}
we consider the decomposition,
\begin{equation*}
Ug(\lambda)= U_1g(\lambda)\:+\:U_2g(\lambda)\:,\:\lambda \in \R,
\end{equation*}
where for $j=1,2$,
\begin{eqnarray*}
U_jg(\lambda):=\frac{\eta_j(\lambda)}{(1+\lambda^2)^{s/2}} \int_{r_1}^{r_2} e^{i\left(\lambda r+t(r)|\lambda|^a\right)}\:g(r)\:d\mu_0(r)\:,\:\:\lambda \in \R\:.
\end{eqnarray*}

We trivially have that
\begin{equation*}
\|U_1g\|_{L^\infty(\R)} \lesssim \|g\|_{L^\infty(d\mu_0)}\:.
\end{equation*}
Now as $Supp(U_1g) \subset (-2,2)$, the above estimate yields 
\begin{equation}\label{concave_pf_eq7}
\|U_1g\|_{L^2(\R)} \lesssim \|g\|_{L^\infty(d\mu_0)}\:.
\end{equation}
Next, to estimate $U_2g$, we start by noting that
\begin{equation*}
\|U_2g\|^2_{L^2(\R)} \lesssim \sum_{k=1}^\infty \|\eta(2^{-k}\cdot)\:Ug\|^2_{L^2(\R)}\:,
\end{equation*}
and then by Fubini's theorem, for each $k \ge 1$,
\begin{eqnarray} \label{sum_estimate}
&& \|\eta(2^{-k}\cdot)\:Ug\|^2_{L^2(\R)} \nonumber\\
&= &\int_{\R}\frac{\eta\left(2^{-k}\lambda\right)^2}{(1+\lambda^2)^s} \int_{r_1}^{r_2}e^{i\left(\lambda r + t(r) |\lambda|^a\right)}\:g(r)\:d\mu_0(r) \int_{r_1}^{r_2}e^{-i\left(\lambda r' + t(r') |\lambda|^a\right)}\:\overline{g(r')}\:d\mu_0(r')\:d\lambda \nonumber\\
\:\:\:\:\:\:\:&\lesssim & \left(2^k\right)^{1-2s} \int_{r_1}^{r_2} \int_{r_1}^{r_2} |g(r)|\:|g(r')|\:|I_{k}(r,r')|\:dr\:dr'\:, 
\end{eqnarray}
where
\begin{equation*}
I_k(r,r'):= \int_\R \eta(\lambda)^2\:e^{i\left(2^k\lambda(r-r')\:+\:2^{ka}|\lambda|^a(t(r)-t(r'))\right)}\:d\lambda\:.
\end{equation*}
For notational convenience, setting $y:=r-r'$ and $t:=t(r)-t(r')$, we estimate the above oscillatory integral by considering the cases:
\begin{equation*}
|y| \le 2^{-\frac{2ks_*}{\alpha_0}}\:,\:\:\text{ or } |y| > 2^{-\frac{2ks_*}{\alpha_0}}\:,
\end{equation*} 
where 
\begin{equation*}
s_*:=\min\left\{\frac{a}{4}+\frac{(1-a)\alpha_0}{2}\:,\: \frac{\alpha_0}{2}\right\}\:,
\end{equation*}
that is, $s=1/2-s_*+\varepsilon$, for any $\varepsilon>0$. In the former case, it follows trivially that
\begin{equation}\label{concave_pf_eq8}
\left|\int_\R \eta(\lambda)^2\:e^{i\left(2^k\lambda y\:+\:2^{ka}|\lambda|^a t)\right)}\:d\lambda\right| \lesssim 1\:.
\end{equation}
We now focus when $|y| > 2^{-\frac{2ks_*}{\alpha_0}}$. In particular, we note that $2^k|y| \ge 1$. For notaional converience, writing, $$\phi(\lambda):= 2^k\lambda y\:+\:2^{ka}|\lambda|^a t\:,$$ we decompose the integral as follows,
\begin{equation*}
\int_{\R} e^{i\phi(\lambda)}\:\eta(\lambda)^2\:d\lambda = \int_{\Omega_1} e^{i\phi(\lambda)}\:\eta(\lambda)^2\:d\lambda \:+\: \int_{\Omega_2} e^{i\phi(\lambda)}\:\eta(\lambda)^2\:d\lambda=: I^{(1)}_k\:+\:I^{(2)}_k\:,
\end{equation*}
where
\begin{eqnarray*}
\Omega_1&:=& \left\{\frac{1}{2} < |\lambda| <2 : 2^{1+ka} |t||\lambda|^{a-1} \ge 2^{4ks_*}|y|^{2\alpha_0}\right\}\:,\\
\Omega_2&:=& \left\{\frac{1}{2} < |\lambda| <2 : 2^{1+ka} |t||\lambda|^{a-1} < 2^{4ks_*}|y|^{2\alpha_0}\right\}\:.
\end{eqnarray*}
We first focus on $I^{(1)}_k$. In this case, since 
\begin{equation*}
\left|\frac{d^2}{d\lambda^2}\phi(\lambda)\right| \gtrsim 2^{4ks_*}|y|^{2\alpha_0}\:,
\end{equation*}
by van der Corput's lemma, we have
\begin{equation}\label{concave_pf_eq9}
\left|I^{(1)}_k\right| \lesssim \left(2^{4ks_*}|y|^{2\alpha_0}\right)^{-\frac{1}{2}} = \left(2^k\right)^{-2s_*}\:|y|^{-\alpha_0}\:.
\end{equation} 
For $I^{(2)}_k$, by the defining condition of $\Omega_2$ and the fact that $t \in \I$, it follows that
\begin{equation*}
2^{ka}\:|t|^{\max\left\{1,2\alpha_0\right\}}\:|\lambda|^{a-1} \le 2^{4ks_*-1}\:|y|^{2\alpha_0}\:.
\end{equation*} 
We also note that
\begin{equation*}
2^{ka}\left(\left(2^k\right)^{4s_*-a}\:|y|^{2\alpha_0}\right)^{\left(\max\left\{1,2\alpha_0\right\}\right)^{-1}} = \left(2^k|y|\right)^{\min\left\{2\alpha_0,1\right\}} \le 2^k|y|\:.
\end{equation*}
Hence, it follows that
\begin{eqnarray*}
\left|\frac{d}{d\lambda}\phi(\lambda)\right| \gtrsim 2^k|y|-a2^{ka}|t||\lambda|^{a-1} &\gtrsim & 2^k|y|-2^{ka}\left(\left(2^k\right)^{4s_*-a}\:|y|^{2\alpha_0}\right)^{\left(\max\left\{1,2\alpha_0\right\}\right)^{-1}} \\
&\gtrsim&  2^k|y|\:,
\end{eqnarray*}
which by van der Corput's lemma yields,
\begin{equation}\label{concave_pf_eq10}
\left|I^{(2)}_k\right| \lesssim \left(2^{k}|y|\right)^{-1} \le \left(2^{k}|y|\right)^{-\alpha_0} \le \left(2^k\right)^{-2s_*}\:|y|^{-\alpha_0}\:.
\end{equation} 
Thus by (\ref{concave_pf_eq8})-(\ref{concave_pf_eq10}), we have
\begin{equation*}
\left|I_k(r,r')\right| \lesssim \chi_{\left\{|r-r'| \le 2^{-\frac{2ks_*}{\alpha_0}}\right\}}(r-r')\:+\: \left(2^k\right)^{-2s_*}\:|r-r'|^{-\alpha_0}\chi_{\left\{|r-r'| > 2^{-\frac{2ks_*}{\alpha_0}}\right\}}(r-r')\:.
\end{equation*}
Thus by Lemma \ref{CS_lemma}, it follows that
\begin{equation*}
\int_{r_1}^{r_2} \int_{r_1}^{r_2} |g(r)|\:|g(r')|\:|I_{k}(r,r')|\:dr\:dr' \lesssim \left(2^k\right)^{-2s_*}\:\|g\|^2_{L^2(d\mu_0)}\:,
\end{equation*}
which upon plugging in (\ref{sum_estimate}) yields for $k \ge 1$\:,
\begin{equation} \label{concave_pf_eq11}
\|\eta(2^{-k}\cdot)\:Ug\|^2_{L^2(\R)} \lesssim \left(2^k\right)^{1-2(s+s_*)}\|g\|^2_{L^2(d\mu_0)} = 2^{-2k\varepsilon}\|g\|^2_{L^2(d\mu_0)} \:.
\end{equation}
Therefore by (\ref{concave_pf_eq7}) and (\ref{concave_pf_eq11}), we get (\ref{concave_pf_eq6}) as,
\begin{eqnarray*}
\|Ug\|_{L^2(\R)} &\le & \|U_1g\|_{L^2(\R)}\:+\:\|U_2g\|_{L^2(\R)} \\ &\lesssim & \|g\|_{L^\infty(d\mu_0)}\:+\:\left(\sum_{k=1}^\infty 2^{-2k\varepsilon}\right)^\frac{1}{2}  \|g\|_{L^2(d\mu_0)} \\
&\lesssim & \|g\|_{L^\infty(d\mu_0)}\:,
\end{eqnarray*}
which along with (\ref{concave_pf_eq5}) yields the desired estimate (\ref{concave_pf_eq4}). Therefore we have proved (\ref{concave_pf_eq1}). 

Then by the method of maximal functions, density of $\Ss(\R^n)$, the fractal Sobolev embedding for the initial data (Lemma \ref{Annulus-Sobolev-embedding}) and Frostmann's lemma, for $a \in (0,1)$, all radial $f \in H^s(\R^n),\:s \in (a/4,1/2)$ and 
\begin{equation} \label{alpha_range}
\alpha > \max \left\{n-\frac{1}{2}\:+\:\frac{1-4s}{2(1-a)}\:,\: n-2s\right\}\:,
\end{equation}
we have
\begin{equation*}
\mathcal{H}^\alpha \left\{x \in \mathbb{A}(r_1,r_2) : \displaystyle\lim_{t \to 0} e^{it(-\Delta)^{a/2}}f(x) \ne f(x)\right\}=0\:.
\end{equation*}
Then by countable additivity of Hausdorff measures, we get that
\begin{equation*}
\mathcal{H}^\alpha \left\{x \in \R^n : \displaystyle\lim_{t \to 0} e^{it(-\Delta)^{a/2}}f(x) \ne f(x)\right\}=0\:,
\end{equation*}
for $\alpha$ satisfying (\ref{alpha_range}). Therefore, we get 
\begin{equation*}
\wtilde{\fd}_{\psi,n}(s) \le \max \left\{n-\frac{1}{2}\:+\:\frac{1-4s}{2(1-a)}\:,\: n-2s\right\}\:,\:\:s \in (a/4,1/2)\:,
\end{equation*}
which is also sharp in the range $s \in [1/4,1/2)$ by Remark \ref{BR_remark}. Remark \ref{BR_remark} also yields the lower bound $\wtilde{\fd}_{\psi,n}(s) \ge n-2s$ for the range $s \in [a/4,1/4)$.

For the range $s>1/2$, the result follows by the above arguments (of the last paragraph) combined with Lemma \ref{lemma_s>1/2}. This completes the proof of part (1) of Theorem \ref{radial_concave_thm}.

\medskip

For part (2) of Theorem \ref{radial_concave_thm}, the result in the range $s \in (0,1/2]$ follows from Walther's work in \cite{Walther2}, whereas for the range $s>1/2$, the result follows by the above arguments combined with Lemma \ref{lemma_s>1/2}. This completes the proof of Theorem \ref{radial_concave_thm}.
\end{proof}

\subsection{Transference principle and applications} \label{subsec6.3}
We now obtain results for families of abstract dispersive PDEs with radial initial data, which can be viewed as natural variations of the results presented in Section \ref{sec5}. The key to this analysis, is again the transference principle:
\begin{proposition} \label{radial_transference_proposition}
The transference principle stated in Theorem \ref{transference_theorem} for dispersive equations with phases of comparable oscillation, continues to hold when restricted to the class of radial initial data.
\end{proposition}
\begin{proof}
The proof of Theorem \ref{transference_theorem} can be carried out verbatim (while incoroporating Lemmata \ref{Annulus-Sobolev-embedding}, \ref{lemma_s>1/2} and Remark \ref{BR_remark}), except that working on $\mathbb{B}(0,1)$ and using the translation invariance of Hausdorff measures, we work on each annuli $\mathbb{A}(r_1,r_2)$ for $0<r_1<r_2<\infty$\:. 
\end{proof}

\begin{corollary} \label{radial_corollary1}
For a dispersive equation with asymptotically convex phase $\psi$ on $\R^n,\:n\ge 2$, i.e. 
\begin{equation*}
\psi(\|\xi\|)=\|\xi\|^a \:+\:\mathcal{O}(1)\:,\:\:\|\xi\| \gg 1\:,\:\:a>1\:,
\end{equation*} we have for radial initial data,
\begin{equation*} 
\wtilde{\fd}_{\psi,n}(s) = \begin{cases}
n\:,\:\:& s<\frac{1}{4} \\
n-2s\:,\:\:& \frac{1}{4} \le s < \frac{1}{2} \\
0\:,\:\:& \frac{1}{2} < s\:. 
\end{cases}
\end{equation*}
\end{corollary}

\begin{remark} \label{radial_beam_remark}
Corollary \ref{radial_corollary1} recovers the result for the Beam propagation \cite[Theorem 2]{LLN} and provides new result for the linearized Boussinesq propagation, for radial initial data. 
\end{remark}

\begin{proof}[Proof of Corollary \ref{radial_corollary1}]
The result follows at once by the result for the fractional Schr\"odinger equations with convex phase given by (\ref{BR_estimates}) and the transference principle Proposition \ref{radial_transference_proposition}. 
\end{proof}

\begin{corollary} \label{radial_corollary2}
For a dispersive equation with asymptotically concave phase $\psi$ on $\R^n,\:n\ge 2$, i.e. 
\begin{equation*}
\psi(\|\xi\|)=\|\xi\|^a \:+\:\mathcal{O}(1)\:,\:\:\|\xi\| \gg 1\:,\:\:a \in (0,1)\:,
\end{equation*}
we have for radial initial data,
\begin{enumerate}
\item the sharp result in the range,
\begin{equation*}
\wtilde{\fd}_{\psi,n}(s) =\begin{cases}
n\:,\:\: &s<\frac{a}{4}\\
n-2s\:,\:\: &\frac{1}{4} \le s < \frac{1}{2}\\
0\:,\:\: &\frac{1}{2} < s\:;
\end{cases}
\end{equation*}
\item and in the remaining range, \\
the upper bound,
\begin{equation*}
\wtilde{\fd}_{\psi,n}(s) \le n-\frac{1}{2}+\frac{1-4s}{2(1-a)}\:,\:\: \frac{a}{4} \le s\le \frac{1}{4}\:;
\end{equation*}
and the lower bound,
\begin{equation*}
\wtilde{\fd}_{\psi,n}(s) \ge n-2s\:,\:\: \frac{a}{4} \le s \le \frac{1}{4}\:.
\end{equation*}
\end{enumerate}
\end{corollary}
\begin{proof}
The result follows by combining the result for the fractional Schr\"odinger equations with concave phase given by part (1) of Theorem \ref{radial_concave_thm} and the transference principle Proposition \ref{radial_transference_proposition}. 
\end{proof}

\begin{corollary} \label{radial_corollary3}
For a dispersive equation with asymptotically linear phase $\psi$ on $\R^n,\:n\ge 2,$ i.e. 
\begin{equation*}
\psi(\|\xi\|)=\|\xi\| \:+\:\mathcal{O}(1)\:,\:\:\|\xi\| \gg 1\:,
\end{equation*}
we have for radial initial data,
\begin{equation*} 
\wtilde{\fd}_{\psi,n}(s)  =\begin{cases}
n\:,\:\: &s \le \frac{1}{2}\\
0\:,\:\: &\frac{1}{2} < s\:.
\end{cases}
\end{equation*}
\end{corollary}
\begin{remark} \label{radial_klein_remark}
As special cases, Corollary \ref{radial_corollary3} provides new results for the family of Klein-Gordon propagations.
\end{remark}

\begin{proof}[Proof of Corollary \ref{radial_corollary3}]
The result follows at once by the result for the Wave propagation given by part (2) of Theorem \ref{radial_concave_thm} and the transference principle Proposition \ref{radial_transference_proposition}. 
\end{proof}

The results entailed in Corollaries \ref{radial_corollary1}, \ref{radial_corollary2} and \ref{radial_corollary3} are presented in Figure \ref{radialfigure2}.
\begin{figure}[htbp]
  \centering
  
  \begin{subfigure}{0.3\textwidth}
    \centering
    \begin{tikzpicture}
\draw [->, thick] (0,2) -- (2,2) node [below left] {$s$};
\draw [->, thick] (0,2) -- (0,7)  node [below left] {$\wtilde{\fd}_{\psi,n}(s)$};
\draw[red] [-, thick] (0,6) -- (0.5,6);
\draw[red][thick] plot[domain=0.5:1, smooth] (\x,{6-2*\x});
\draw[red] [-, thick] (1,2) -- (2,2);
\draw [fill] (1,2) circle [radius=.05] node [below] {\scriptsize$\frac{1}{2}$};
\draw [fill] (0.5,2) circle [radius=.05] node [below] {\scriptsize$\frac{1}{4}$};
\draw [fill] (0,6) circle [radius=.05] node [left] {\scriptsize$n$};
\draw [fill] (0,5) circle [radius=.05] node [left] {\scriptsize$n-\frac{1}{2}$};
\draw [fill] (0,4) circle [radius=.05] node [left] {\scriptsize$n-1$};
\draw[gray] [-, dashed] (1,2) -- (1,4);
\draw[gray] [-, dashed] (0.5,2) -- (0.5,6);
\draw[gray] [-, dashed] (0,4) -- (1,4);
\draw[gray] [-, dashed] (0,5) -- (0.5,5);
\end{tikzpicture}
    \caption{asymptotically convex phase}
    \label{radialfigure2.1}
  \end{subfigure}
  \hfill 
  \begin{subfigure}{0.3\textwidth}
    \centering
    \begin{tikzpicture}
\draw [->, thick] (0,2) -- (2,2) node [below left] {$s$};
\draw [->, thick] (0,2) -- (0,7)  node [below left] {$\wtilde{\fd}_{\psi,n}(s)$};
\draw[red] [-, thick] (0,6) -- (0.25,6);
\draw[red][thick] plot[domain=0.5:1, smooth] (\x,{6-2*\x});
\draw[red] [-, thick] (1,2) -- (2,2);
\draw[cyan][thick]plot[domain=0.25:0.5, smooth] (\x,{7-4*\x});
\draw[green][thick]plot[domain=0.25:0.5, smooth] (\x,{6-2*\x});
\draw [fill] (1,2) circle [radius=.05] node [below] {\scriptsize$\frac{1}{2}$};
\draw [fill] (0.5,2) circle [radius=.05] node [below] {\scriptsize$\frac{1}{4}$};
\draw [fill] (0.25,2) circle [radius=.05] node [below] {\scriptsize$\frac{a}{4}$};
\draw [fill] (0,6) circle [radius=.05] node [left] {\scriptsize$n$};
\draw [fill] (0,4) circle [radius=.05] node [left] {\scriptsize$n-1$};
\draw [fill] (0,5) circle [radius=.05] node [left] {\scriptsize$n-\frac{1}{2}$};
\draw [fill] (0,5.5) circle [radius=.05] node [left] {\scriptsize$n-\frac{a}{2}$};
\draw[gray] [-, dashed] (0.25,2) -- (0.25,6);
\draw[gray] [-, dashed] (0,5.5) -- (0.25,5.5);
\draw[gray] [-, dashed] (0.5,2) -- (0.5,5);
\draw[gray] [-, dashed] (1,2) -- (1,4);
\draw[gray] [-, dashed] (0,4) -- (1,4);
\draw[gray] [-, dashed] (0,5) -- (0.5,5);
\draw[red] [-, thick] (1.5,6) -- (1.7,6);
\pgftext[at={\pgfpoint{3cm}{6cm}}]{\textcolor{red}{sharp result}};
\draw[cyan] [-, thick] (1.5,5.5) -- (1.7,5.5);
\pgftext[at={\pgfpoint{3cm}{5.5cm}}]{\textcolor{cyan}{upper bound}};
\draw[cyan] [-, thick] (1.5,5) -- (1.7,5);
\pgftext[at={\pgfpoint{3cm}{5cm}}]{\textcolor{green}{lower bound}};	
\end{tikzpicture}
    \caption{asymptotically concave phase}
    \label{radialfigure2.2}
  \end{subfigure}
  \hfill 
  \begin{subfigure}{0.3\textwidth}
    \centering
    \begin{tikzpicture}
\draw [->, thick] (0,2) -- (2,2) node [below left] {$s$};
\draw [->, thick] (0,2) -- (0,7)  node [below left] {$\wtilde{\fd}_{\psi,n}(s)$};
\draw[red] [-, thick] (0,6) -- (1,6);
\draw[red] [-, thick] (1,2) -- (2,2);
\draw [fill] (1,2) circle [radius=.05] node [below] {\scriptsize$\frac{1}{2}$};
\draw [fill] (0,6) circle [radius=.05] node [left] {\scriptsize$n$};
\draw[gray] [-, dashed] (1,2) -- (1,6);
\end{tikzpicture}
    \caption{asymptotically linear phase}
    \label{radialfigure2.3}
  \end{subfigure}
  \caption{$\wtilde{\fd}_{\psi,n}$ for general classes of dispersive PDEs}
  \label{radialfigure2}
\end{figure}
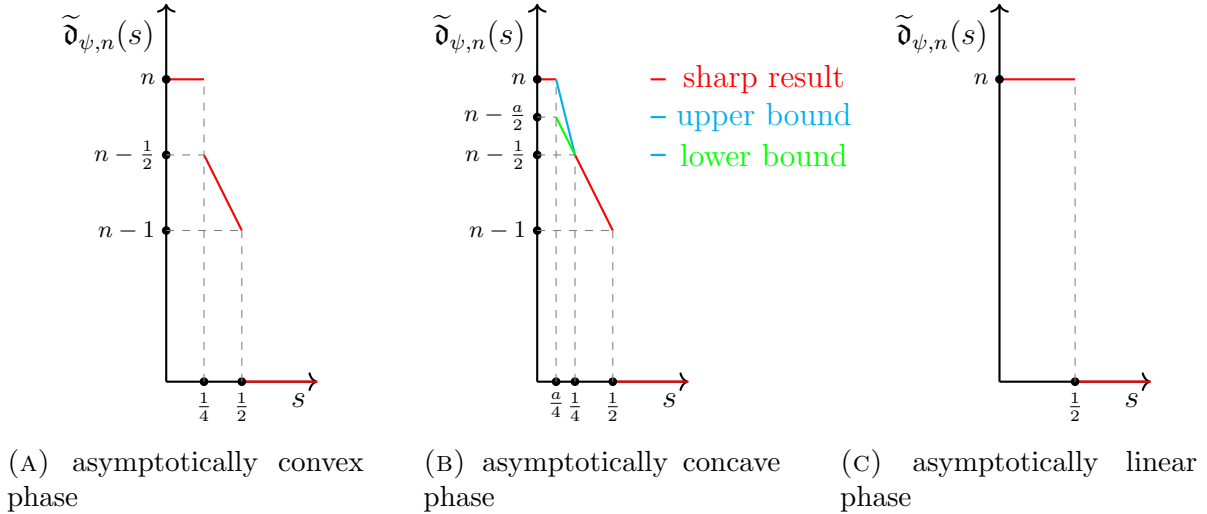
\section{Results on the $3$-dimensional hyperbolic space}
\label{sec7}
In this section, we study the divergence sets of dispersive equations on $\Hb^3$, the three-dimensional real hyperbolic space. Besides being of constant negative sectional curvature, $\Hb^3$ also enjoys the homogeneous structure of a Riemannian symmetric space. We now recall several standard notions in Harmonic Analysis on Riemannian symmetric spaces for $\Hb^3$. For unexplained terminologies and more details, we refer to \cite{Helgason}.

Let $G=SL(2,\C)$ and $K$ be its maximal compact subgroup $SU(2)$. We then have the identification,
\begin{equation*}
G/K \cong \mathbb{H}^3\:,
\end{equation*}
with constant sectional curvature $-1$. This identification allows one to define convolution on $\Hb^3$ via $G$. The $G$-invariant Riemannian measure on $G/K$ is $\sinh^2 r \: dr \: dk$, where $dr$ and $dk$ are the Lebesgue measure on $(0,\infty)$ and the Haar measure on $SU(2)$ respectively. Thus the volumes of geodesic balls satisfy the asymptotics:
\begin{equation} \label{volume_asymptotics}
vol(B(x,r)) \asymp \begin{cases}
r^3\:,\:\:&r\in(0,1)\\
e^{2r}\:,\:\:&r\ge 1\:.
\end{cases}
\end{equation}

A function $f: \Hb^3 \to \C$ is said to be {\em geodesically radial} if, for all $x$ in $\Hb^3$, $f(x)$ depends only on the geodesic distance of $x$ from the origin $o$. Henceforth, we will refer to geodesically radial functions simply as {\em radial functions}. 

We now recall the spherical functions on $\Hb^3$. The spherical functions $\varphi_\lambda$ on $\Hb^3$, for $\lambda \in \C$ are the radial eigenfunctions of the Laplace-Beltrami operator $\Delta_{\Hb^3}$, satisfying the following normalization criterion
\begin{equation*}
\begin{cases}
 & \Delta \varphi_\lambda = - \left(\lambda^2 + 1\right) \varphi_\lambda\:,  \\
& \varphi_\lambda(o)=1 \:.
\end{cases}
\end{equation*}
Since $G$ is a complex, non-compact semi-simple Lie group, a speciality of $\Hb^3$ is the explicitness of the spherical functions \cite[p. 432, Theorem 5.7]{Helgason}: 
\begin{equation*}
\varphi_\lambda(r) = \frac{\sin(\lambda r)}{\lambda \sinh(r)}\:,\:\:\lambda \in (0,\infty)\:.
\end{equation*}
The spherical functions are crucial as they help us define the Spherical Fourier transform of a ``nice" radial function $f$ (on $\Hb^3$) in the following way:
\begin{equation*}
\mathscr{F}f(\lambda):= \int_{\Hb^3} f(x)\: \varphi_\lambda(x)\: dvol(x) = \int_0^\infty f(r) \:\varphi_\lambda(r) \:\sinh^2(r)\:  dr\:.
\end{equation*}
One has the following inversion formula (when valid) for radial functions:
\begin{equation*}
f(r)= \int_{0}^\infty \mathscr{F}f(\lambda)\:\varphi_\lambda(r)\: \lambda^2 \: d\lambda\:.
\end{equation*}
Moreover, the Spherical Fourier transform extends to an isometry from the space of radial $L^2$ functions on $\Hb^3$ onto $L^2\left((0,\infty),\lambda^2 d\lambda\right)$. Due to the spectral shift of $\Delta_{\Hb^3}$, we consider the following $L^2$-fractional Sobolev spaces specialized to radial functions:
\begin{equation*} 
H^s(\Hb^3)_{radial}:=\left\{f \in L^2(\Hb^3)_{radial}: {\|f\|}_{H^s(\Hb^3)}:= {\left(\int_0^\infty {\left(\lambda^2 + 1\right)}^s {|\mathscr{F}f(\lambda)|}^2 \:\lambda^2 d\lambda\right)}^{\frac{1}{2}}< \infty\right\}\:.
\end{equation*}
The class of radial $L^2$-Schwartz class functions on $\Hb^3$, denoted by $\Ss^2(\Hb^3)_{radial}$, is defined to be the collection of $f \in C^\infty(\Hb^3)_{radial}$ such that 
\begin{equation*} 
\left|{\left(\frac{d}{dr}\right)}^M f(r)\right| \lesssim {(1+r)}^{-N} e^{-r}\:, \text{ for any } M, N \in \N \cup \{0\}\:.
\end{equation*}

\subsection{Two correspondence lemmata} \label{subsec7.1} 
We now present two correspondence results between $\Hb^3$ and $\R^3$, that will be the building blocks for our forthcoming arguments.

Our first result is geometric in nature, which we exhibit in greater generality. Let $\X$ be a non-positively curved Hadamard manifold with a fixed origin $o$. We now present a correspondence between fractal measures on geodesic annuli $A(r_1,r_2)$ with center $o$ in $\X$  and on annuli $\mathbb{A}(r_1,r_2)$  with center $0$, with respect to the flat metric in the tangent space $T_o\X$:
\begin{lemma} \label{fractal_measure_correspondence_lemma}
On a non-positively curved Hadamard manifold, for any $0<r_1<r_2<\infty$ and $\alpha \ge 0$, the Riemannian exponential map $\exp_o$ induces an one-to-one correspondence between $\M^\alpha\left(\mathbb{A}(r_1,r_2)\right)$ and $\M^\alpha\left(A(r_1,r_2)\right)$. 
\end{lemma}
\begin{proof}
We start off by making the canonical identification $T_o\X \cong \R^n$, where $n$ is the topological dimension of $\X$. Due to the underlying geometry of $\X$,  $\exp_o$ is a diffeomorphism from $\R^n$ onto $\X$ and is also a radial isometry. Hence $\exp_o\left(\mathbb{A}(r_1,r_2)\right)=A(r_1,r_2)$ and moreover, it suffices to prove that for any $\mu \in \M^\alpha\left(\mathbb{A}(r_1,r_2)\right)$, its pushforward,
\begin{equation*}
(\left(\exp_o\right)_*\mu)(B):=\mu\left(\exp^{-1}_o(B)\right)\:,\:\:B \in \mathscr{B}(\X)\:,
\end{equation*}
is a member of $\M^\alpha\left(A(r_1,r_2)\right)$. Let $\mu \in \M^\alpha\left(\mathbb{A}(r_1,r_2)\right)$. Then there exists $C_1>0$ such that 
\begin{equation} \label{euclidean_inequality}
\mu\left(\mathbb{B}(x,r)\right) \le C_1r^\alpha\:,\:\:\text{ for all } x \in \R^n\:,\:r>0\:.
\end{equation}

It is clear that $\left(\exp_o\right)_*\mu$ is a positive Radon measure on $\X$ with support contained in $A(r_1,r_2)$. We now consider $(\left(\exp_o\right)_*\mu)(B(x,r))$ for geodesic balls $B(x,r) \subset \X$. For $r \ge 1$, we trivially have,
\begin{equation} \label{fractal_correspondence_eq1}
(\left(\exp_o\right)_*\mu)(B(x,r))=\mu\left(\exp^{-1}_o(B(x,r))\right) \le \mu\left(\mathbb{A}(r_1,r_2)\right) \le \mu\left(\mathbb{A}(r_1,r_2)\right) r^\alpha\:. 
\end{equation}
We next consider the case $r \in (0,1)$ and pay attention to the location of the center $x$. If $d(o,x) \ge r_2+1$, then $B(x,r) \cap  A(r_1,r_2)=\emptyset$ and hence,
\begin{equation} \label{fractal_correspondence_eq2}
(\left(\exp_o\right)_*\mu)(B(x,r))=0 \le r^\alpha\:.
\end{equation}
Finally if $d(o,x)<r_2+1$, then we note that $B(x,r) \subset B(o,r_2+2)$. Since $\exp_o$ is diffeomorphism from $\R^n$ onto $\X$, it is bi-Lipschitz from $\mathbb{B}(0,r_2+2)$ onto $B(o,r_2+2)$, that is, there exists a constant $C_2\ge 1$ such that 
\begin{equation*}
\frac{1}{C_2} \|a-b\| \le d(\exp_o(a),\exp_o(b)) \le C_2 \|a-b\|\:,\:\:a,b\in \mathbb{B}(0,r_2+2)\:.
\end{equation*}
Thus $$\exp^{-1}_o\left(B(x,r)\right) \subset B\left(\exp^{-1}_o(x),C_2r\right)\:,$$
which in turn by (\ref{euclidean_inequality}) yields
\begin{equation} \label{fractal_correspondence_eq3}
(\left(\exp_o\right)_*\mu)(B(x,r)) \le \mu\left(B\left(\exp^{-1}_o(x),C_2r\right)\right) \le (C_1 C^\alpha_2)\:r^\alpha\:.
\end{equation}
Then by (\ref{fractal_correspondence_eq1})-(\ref{fractal_correspondence_eq3}), the result follows by considering $\max\left\{\mu\left(\mathbb{A}(r_1,r_2)\right),\:1,\:C_1 C^\alpha_2\right\}$.
\end{proof}

Our next result is of analytic flavour and is specialized for $\Hb^3$:
\begin{lemma} \label{schwartz_correspondence_lemma}
Given any radial $f \in \mathcal{S}^2(\Hb^3)$ there exists a unique radial $F \in \mathcal{S}(\R^3)$ (and vice-versa) such that 
\begin{itemize}
\item $\mathscr{F}f=\what{F}$ on $[0,\infty)$ and 
\item $\|f\|_{H^s(\Hb^3)}\asymp\|F\|_{H^s(\R^3)}$, for all $s \ge 0$\:.
\end{itemize}
\end{lemma}

\begin{proof}
By the well-known properties of the Abel transform on Ch\'ebli-Trim\`eche hypergroups \cite{BX} (for symmetric spaces see \cite{AnkerJFA} and also \cite[Lemma 2.3]{Helgason-b}), we have the following commutative diagram, where
\begin{itemize}
\item $\mathscr{A}_1$ is the Abel transform defined from $\mathcal{S}^2(\Hb^3)_{radial}$ onto the space of even Schwartz class functions on $\R$, ${\mathcal{S}(\R)}_{even}$\:.
\item $\mathscr{A}_2$ is the Abel transform defined from the space of radial Schwartz class functions on $\R^3$, $\mathcal{S}(\R^3)_{radial}$ onto ${\mathcal{S}(\R)}_{even}$\:.
\item $\mathscr{F}$ denotes the Spherical Fourier transform from $\mathcal{S}^2(\Hb^3)_{radial}$ onto ${\mathcal{S}(\R)}_{even}$\:.
\item $\wedge$ denotes the Euclidean Fourier transform from $\mathcal{S}(\R^3)_{radial}$ onto ${\mathcal{S}(\R)}_{even}$\:.
\item $\sim$ denotes the 1-dimensional Euclidean Fourier transform from ${\mathcal{S}(\R)}_{even}$ onto itself.
\end{itemize}

\[
\begin{tikzcd}[row sep=1.4cm,column sep=1.4cm]
\mathcal{S}^2(\Hb^3)_{radial} \arrow[r,"\mathscr{A}_1"] \arrow[dr,swap,"\mathscr{F}"] & {\mathcal{S}(\R)}_{even}
\arrow[d,"\sim"] &
\arrow[l,"\mathscr{A}_2",swap]  \mathcal{S}(\R^3)_{radial} \arrow[dl,"\wedge"] \\
&  {\mathcal{S}(\R)}_{even}&
\end{tikzcd}
\]

Then, since all the maps above are topological isomorphisms, defining $\mathscr{A}:=\mathscr{A}^{-1}_2 \circ \mathscr{A}_1$, one can reduce matters to the following simplified commutative diagram:

\[
\begin{tikzcd}[row sep=1.4cm,column sep=1.4cm]
\mathcal{S}^2(\Hb^3)_{radial}\arrow[r,"\mathscr{A}"] \arrow[dr,swap,"\mathscr{F}"] & \mathcal{S}(\R^3)_{radial}
\arrow[d,"\wedge"]   \\
&  {\mathcal{S}(\R)}_{even}&
\end{tikzcd}
\]

Hence, for $f \in \mathcal{S}^2(\Hb^3)_{radial}$, we consider $F:= \mathscr{A}f \in \mathcal{S}(\R^3)_{radial}$ to observe that  
\begin{equation*}
\mathscr{F}f(\lambda)= \what{F}(\lambda)\:,\:\:\lambda \ge 0\:.
\end{equation*}
Moreover, for any $s \ge 0$, 
\begin{eqnarray*} 
{\|f\|}_{H^s(\Hb^3)} &=& {\left(\int_0^\infty {(\lambda^2 + 1)}^s \:{\left|\mathscr{F}f(\lambda)\right|}^2 \: \lambda^2 \: d\lambda\right)}^{1/2} \\
&=& {\left(\int_0^\infty {(\lambda^2 + 1)}^s \:{\left|\what{F}(\lambda)\right|}^2 \: \lambda^2 \: d\lambda\right)}^{1/2} \nonumber\\
&\asymp&  {\|F\|}_{H^s(\R^3)}\:.
\end{eqnarray*}
This completes the proof of Lemma \ref{schwartz_correspondence_lemma}.
\end{proof}

\subsection{Fractal Sobolev embedding and a fractal maximal estimate in high regularity} \label{subsec7.2}
As in the classical Euclidean setting, in order to make the notion of the Hausdorff dimension of divergence sets well-defined, we need to fix a canonical choice of representative for the intial data and its propagation that will make sense even on the fractal level.

For $x \in \Hb^3$, we denote $d(o,x)$ by $r$. For initial data $f \in \mathcal{S}^2(\Hb^3)_{radial}$, the solution to the dispersive equation
\begin{equation*} 
\begin{cases}
	 i\frac{\partial u}{\partial t} +\psi(\sqrt{-\Delta_{\Hb^3}} )u=0\:,\:  &(x,t) \in \Hb^3 \times \R\:, \\
	u(\cdot,0)=f\:,\: &\text{ on } \Hb^3 \:,
	\end{cases}
\end{equation*}
is expressed pointwise by the spherical Fourier inversion,
\begin{equation*}
u(r,t)=\int_0^\infty \varphi_\lambda(r)\:e^{it\psi(\sqrt{\lambda^2+1})}\:\mathscr{F}f(\lambda)\:\lambda^2\:d\lambda\:,
\end{equation*}
and is continuous. But for general radial $f \in H^s(\Hb^3),$ we again consider
\begin{equation*}
S^N_{\psi,t}f(r):= \int_0^\infty \Phi\left(N^{-1}\lambda\right)\: \varphi_\lambda(r)\:e^{it\psi(\sqrt{\lambda^2+1})}\:\mathscr{F}f(\lambda)\:\lambda^2\:d\lambda\:,\:\:N\in \N\:,
\end{equation*}
where $\Phi(\lambda):=e^{-\lambda^2}$ and then define $u$ as the pointwise limit (whenever it exists),
\begin{equation} \label{H3_representative}
u(\cdot,t):=\lim_{N \to \infty} S^N_{\psi,t}f\:.
\end{equation} 
It is easy to see that, the choice of $u(\cdot,t)$ given by (\ref{H3_representative}) coincides with the classical $L^2$-limits, in the sense of volume measure. But interestingly, it is also well-defined with respect to fractal measures for any $s>0$. The key to this is the following estimate: 
\begin{lemma} \label{H3-N-maximal-estimate-radial}
Let $s>0$ and $\alpha >3-2s$. Then for any radial $f \in H^s(\Hb^3)$, any $\mu \in \M^\alpha(A(r_1,r_2))$, for $0<r_1<r_2<\infty$ and any $t \in \R$, we have
\begin{equation*}
\left\|\sup_{N \ge 1} \left|S^N_{\psi,t}f\right|\right\|_{L^2(d\mu)} \lesssim \|f\|_{H^s(\Hb^3)}\:,
\end{equation*}
where the implicit constant depends only on $\alpha, \mu, r_1,r_2$\:.
\end{lemma}
\begin{proof}
We fix $s>0,\:\alpha >3-2s$ and $\mu \in \M^\alpha(A(r_1,r_2))$ with its one-dimensional radial pushforward denoted by $\mu_0$. By density, it suffices to prove the inequality for all radial functions in $\mathcal{S}^2(\Hb^3)$. Let $f \in \mathcal{S}^2(\Hb^3)_{radial}$. By Lemma \ref{schwartz_correspondence_lemma}, there exists $F \in \mathcal{S}(\R^3)_{radial}$ such that 
\begin{equation*}
\mathscr{F}f=\what{F}\:,\:\:\text{ on } [0,\infty) \text{ and } 
\|f\|_{H^s(\Hb^3)}\asymp\|F\|_{H^s(\R^3)}\:\:.
\end{equation*}
Thus for all $r \in (r_1,r_2)$ and all $t \in \R$,
\begin{eqnarray} \label{H3-N-eq1}
S^N_{\psi,t}f(r)&=& \int_0^\infty \Phi\left(N^{-1}\lambda\right)\: \varphi_\lambda(r)\:e^{it\psi(\sqrt{\lambda^2+1})}\:\mathscr{F}f(\lambda)\:\lambda^2\:d\lambda \nonumber\\
&=& \int_0^\infty \Phi\left(N^{-1}\lambda\right)\:\frac{\sin(\lambda r)}{\lambda \sinh(r)} \:e^{it\psi(\sqrt{\lambda^2+1})}\:\mathscr{F}f(\lambda)\:\lambda^2\:d\lambda \nonumber\\
&=& \left(\frac{ r}{\sinh(r)}\right) \int_0^\infty \Phi\left(N^{-1}\lambda\right)\: \frac{\sin(\lambda r)}{\lambda r} \: e^{it\wtilde{\psi}(\lambda)} \: \what{F}(\lambda) \: \lambda^2 \: d\lambda \nonumber\\
&=& \left(\frac{ r}{\sinh(r)}\right) \int_0^\infty \Phi\left(N^{-1}\lambda\right)\left( \frac{J_{\frac{1}{2}}(\lambda r)}{{(\lambda r)}^{\frac{1}{2}}} \right) e^{it\wtilde{\psi}(\lambda)} \: \what{F}(\lambda) \: \lambda^2 \: d\lambda \nonumber\\
&=& C\left(\frac{ r}{\sinh(r)}\right) S^N_{\wtilde{\psi},t}F(r)\:,
\end{eqnarray}
for some $C>0$, where $S^N_{\wtilde{\psi},t}F$ is the Euclidean counter-part given by (\ref{N-scale}), corresponding to the dispersive equation:
\begin{equation*} 
\begin{cases}
	 i\frac{\partial u}{\partial t} +\wtilde{\psi}(\sqrt{-\Delta} )u=0\:,\:  &(x,t) \in \mathcal \R^3 \times \R\:, \\
	u(\cdot,0)=F\:,\: &\text{ on } \mathcal \R^3 \:,
	\end{cases}
\end{equation*}
where $\wtilde{\psi}(\lambda):=\psi(\sqrt{\lambda^2+1})$. Thus by (\ref{H3-N-eq1}), we have for all $r \in (r_1,r_2)$ and all $t \in \R$, the pointwise comparability,
\begin{equation} \label{H3-N-eq2}
\sup_{N \ge 1} \left|S^N_{\psi,t}f(r)\right| \asymp \sup_{N \ge 1} \left|S^N_{\wtilde{\psi},t}F(r)\right|\:.
\end{equation}
Now by Lemma \ref{fractal_measure_correspondence_lemma}, identifying $T_o\Hb^3 \cong \R^3$, there exists $\nu \in \M^\alpha(\mathbb{A}(r_1,r_2))$ such that $\mu=(\exp_o)_*\nu$. Moreover, if the one-dimensional radial pushforward of $\nu$ is given by $\nu_0$, then by the fact that $\exp_o$ is a radial isometry, it follows that, $\mu_0=\nu_0$, as measures on $[0,\infty)$, supported in the interval $(r_1,r_2)$. Combining this with (\ref{H3-N-eq2}), we have,
\begin{eqnarray*}
\left\|\sup_{N \ge 1} \left|S^N_{\psi,t}f\right|\right\|_{L^2(d\mu)} &=& \left(\int_{r_1}^{r_2}\left(\sup_{N \ge 1} \left|S^N_{\psi,t}f(r)\right|\right)^2\:d\mu_0(r)\right)^\frac{1}{2} \\
&\asymp & \left(\int_{r_1}^{r_2}\left(\sup_{N \ge 1} \left|S^N_{\wtilde{\psi},t}F(r)\right|\right)^2\:d\nu_0(r)\right)^\frac{1}{2} \\
&=& \left\|\sup_{N \ge 1} \left|S^N_{\wtilde{\psi},t}F\right|\right\|_{L^2(d\nu)}\:,
\end{eqnarray*}
which along with Lemma \ref{N-maximal-estimate-radial} yields
\begin{eqnarray*}
\left\|\sup_{N \ge 1} \left|S^N_{\psi,t}f\right|\right\|_{L^2(d\mu)}  \lesssim  \left\|\sup_{N \ge 1} \left|S^N_{\wtilde{\psi},t}F\right|\right\|_{L^2(d\nu)} \lesssim \|F\|_{H^s(\R^3)} \lesssim \|f\|_{H^s(\Hb^3)}\:.
\end{eqnarray*}
Since $f$ was arbitrary, this completes the proof of Lemma \ref{H3-N-maximal-estimate-radial}.
\end{proof}

\begin{remark} \label{inter-spatial-remark}
The inter-spatial transference argument for the fractal $N$-maximal estimates employed in the proof of Lemma \ref{H3-N-maximal-estimate-radial}, combining the building blocks of Lemmata \ref{fractal_measure_correspondence_lemma} and \ref{schwartz_correspondence_lemma}, is the heart of the matter of this section and is of independent interest. For instance, the above idea along with the arguments in Appendix A in \cite{BBCR} also yields that the choice given by (\ref{H3_representative}) for the initial data $f$ also coincides up to a set of Hausdorff dimension $\le 3-2s$, with the classical choice given as a Bessel potential. Moreover, it allows us to obtain the transference principle Theorem \ref{H3_transference} (in Subsection \ref{subsec7.4}), leading up to the proof of Theorem \ref{H3_theorem}.
\end{remark}

As a consequence of Lemma \ref{H3-N-maximal-estimate-radial}, we have the following  fixed-time fractal Sobolev embedding on geodesic annuli which garners the fractal well-definedness:
\begin{lemma} \label{H3-Sobolev-embedding}
Let $s>0$ and $\alpha >3-2s$. Then for any radial $f \in H^s(\Hb^3)$, any $\mu \in \M^\alpha(A(r_1,r_2))$, for $0<r_1<r_2<\infty$ and any $t \in \R$, we have
\begin{equation*}
\left\|u(\cdot,t)\right\|_{L^2(d\mu)} \lesssim \|f\|_{H^s(\Hb^3)}\:,
\end{equation*}
where the implicit constant depends only  on $\alpha, \mu,r_1,r_2$\:.
\end{lemma}

In high regularity $s>1/2$ for radial initial data, however,  the situation is much better:
\begin{lemma} \label{H3-lemma_s>1/2}
We have for all phases $\psi$, all $f \in H^s(\Hb^3)$ radial, for $s>1/2$,
\begin{equation*}
\left\|\sup_{t \in \I} \sup_{N \ge 1} \left|S^N_{\psi,t}f\right|\right\|_{L^1(d\mu)} \lesssim \|f\|_{H^s(\Hb^3)}\:,
\end{equation*}
where $\mu \in \M^\alpha(A(r_1,r_2))$, for $0<r_1<r_2<\infty,\:\alpha>0$. The implicit constant in the above inequality depends only on $\mu,\:r_1,\:r_2$\:.
\end{lemma} 
\begin{proof}
The proof follows by the similar inter-spatial transference argument as in the proof of Lemma \ref{H3-N-maximal-estimate-radial}, combining with Lemma \ref{lemma_s>1/2}.
\end{proof}

\subsection{Singularity of Sobolev functions} \label{subsec7.3}
In this subsection, we will obtain the following quantitative result on the singularities of Sobolev functions:
\begin{theorem} \label{sobolev_singularity_theorem}
\begin{enumerate}
\item Let $s \in (0,3/2)$. Then there exists a compact subset $\mathcal{K} \subset \Hb^3$ and a function $f \in H^s(\Hb^3)$ such that $dim_\mathcal{H}(\mathcal{K})<3-2s$ and $f$ is singular precisely on $\mathcal{K}$.
\item Let $s \in (0,1/2)$. Then there exists a radially symmetric compact subset $\mathcal{K} \subset \Hb^3$ and a radial $f \in H^s(\Hb^3)$ such that $dim_\mathcal{H}(\mathcal{K})<3-2s$ and $f$ is singular precisely on $\mathcal{K}$.
\end{enumerate}
\end{theorem}

Before commencing the proof of Theorem \ref{sobolev_singularity_theorem}, we will need the following result for the Bessel kernel on $\Hb^3$:
\begin{lemma} \label{lemma_bessel_kernel}
The Bessel kernel of order $s>0$ on $\Hb^3$ is radial and has the explicit form:
\begin{equation*}
\mathcal{G}_s(r):= C\:\left(\frac{r}{\sinh r}\right)\:r^{\frac{s-3}{2}}\:K_{\frac{3-s}{2}}(r)\:,
\end{equation*}
where $K_\nu$ is the modified Bessel function of the second kind of order $\nu$ and $C$ is a positive constant depending only on $s$. Consequently, the Bessel kernel satisfies the following asymptotics,
\begin{equation*}
\mathcal{G}_s(r) \asymp \begin{cases}
r^{s-3}\:,\:\:&r\in(0,1)\:,\\
r^{\frac{s-2}{2}}\:e^{-2r}\:,\:\:&r \ge 1\:.
\end{cases}
\end{equation*}
\end{lemma}  
\begin{proof}
By the heat semi-group representation of $(-\Delta_{\Hb^3})^{-s/2}$, we have,
\begin{equation*}
(-\Delta_{\Hb^3})^{-s/2} = \frac{1}{\Gamma(s/2)} \int_0^\infty t^{\frac{s}{2}-1}\:e^{t\Delta_{\Hb^3}}\:dt\:.
\end{equation*}
Now as $(-\Delta_{\Hb^3})^{-s/2}f$ for an $f \in C^\infty_c(\Hb^3)$ is given by a right convolution with the radial kernel $\mathcal{G}_s$, we get
\begin{equation} \label{bessel_eq1}
\mathcal{G}_s(r) = \frac{1}{\Gamma(s/2)} \int_0^\infty t^{\frac{s}{2}-1}\:h_t(r)\:dt\:,
\end{equation}
where $h_t$ is the heat kernel of $\Hb^3$, explicitly given by \cite{AO},
\begin{equation*}
h_t(r):=ct^{-\frac{3}{2}}\:e^{-t}\:\left(\frac{r}{\sinh r}\right)\:e^{-\frac{r^2}{4t}}\:.
\end{equation*}
Plugging the last expression in (\ref{bessel_eq1}), we have
\begin{equation*} 
\mathcal{G}_s(r) = \frac{c}{\Gamma(s/2)} \left(\frac{r}{\sinh r}\right)\int_0^\infty t^{\frac{s-3}{2}-1}\:e^{-\left(t+\frac{r^2}{4t}\right)}\:dt\:.
\end{equation*}
Then performing the change of variable: $t=(r/2)u$, we get
\begin{equation*} 
\mathcal{G}_s(r) = \frac{c}{\Gamma(s/2)} \left(\frac{r}{\sinh r}\right) \left(\frac{r}{2}\right)^{\frac{s-3}{2}}\int_0^\infty u^{\frac{s-3}{2}-1}\:e^{-\frac{r}{2}\left(u+\frac{1}{u}\right)}\:du\:.
\end{equation*}
Identifying the last integral as the modified Bessel function of the second kind of the appropriate order, we get the result.
\end{proof}

We now present the proof of Theorem \ref{sobolev_singularity_theorem}:
\begin{proof}[Proof of Theorem \ref{sobolev_singularity_theorem}]
We first focus on part (1). We fix $s \in (0,3/2)$ and choose a compact subset $\mathcal{K}' \subset \R^3$ such that 
\begin{equation*}
dim_{\mathcal{H}}(\mathcal{K}')=dim_{\mathcal{M}}(\mathcal{K}')<3-2s\:,
\end{equation*}
where $dim_{\mathcal{M}}(\mathcal{K}')$ is the Minkowski dimension of $\mathcal{K}'$. For example, if the dimension is integer, we can simply take a compact line segment or a compact 2-disk, whereas, for a non-integer dimension say $\lambda$, we write $\lambda=\zeta+\lfloor \lambda \rfloor$, $\zeta \in (0,1)$ and consider $\mathcal{K}'=\mathcal{C} \times [0,1]^{\lfloor \lambda \rfloor}$ where $\mathcal{C}$ is a generalized Cantor set with both Hausdorff and Minkowski dimension equal to $\zeta$. 

By identifying $T_o\Hb^3 \cong \R^3$, we then consider
\begin{equation*}
\mathcal{K}:=\exp_o(\mathcal{K}')\:.
\end{equation*}
Since $\exp_o$ is locally bi-Lipschitz, we get that
\begin{equation*}
dim_{\mathcal{H}}(\mathcal{K})=dim_{\mathcal{M}}(\mathcal{K})=dim_{\mathcal{H}}(\mathcal{K}')=dim_{\mathcal{M}}(\mathcal{K}')<3-2s\:.
\end{equation*}

We choose and fix $\alpha \in \left(dim_{\mathcal{M}}(\mathcal{K})\:,\:3-2s\right)$. For $\varepsilon>0$, we consider the $\varepsilon$-neighborhood of $\mathcal{K}$,
\begin{equation*}
\mathcal{K}_\varepsilon:=\left\{x \in \Hb^3 : d(x,\mathcal{K})<\varepsilon\right\}\:.
\end{equation*}
Next for $\gamma \in \left(s,(3-\alpha)/2\right)$, let us define, $g : \Hb^3 \to [0,+\infty]$ by,
\begin{equation*}
g(x):=\begin{cases}
d(x,\mathcal{K})^{-\gamma}\:,\:\:& x \in \mathcal{K}_\varepsilon\\
0\:,\:\:& x \in \Hb^3 \setminus \mathcal{K}_\varepsilon\:.
\end{cases}
\end{equation*}
Now since, $\alpha > dim_{\mathcal{M}}(\mathcal{K})$, the $\alpha$-dimensional Minkowski content of $\mathcal{K}$ defined as,
\begin{equation*}
\mathscr{M}^\alpha(\mathcal{K}):=\lim_{\varepsilon \to 0+} \frac{vol\left(\mathcal{K}_\varepsilon\right)}{\varepsilon^{3-\alpha}}\:,
\end{equation*}
is $0$ and thus
\begin{eqnarray*}
\|g\|^2_{L^2(\Hb^3)} &=& \int_{\mathcal{K}_\varepsilon} d(x,\mathcal{K})^{-2\gamma}\:dvol(x) \\
& \le & \sum_{j=0}^\infty  \int_{\mathcal{K}_{2^{-j}\varepsilon} \setminus \mathcal{K}_{2^{-j-1}\varepsilon}} d(x,\mathcal{K})^{-2\gamma}\:dvol(x) \\
& \le & \sum_{j=0}^\infty \left(2^{-j-1}\varepsilon\right)^{-2\gamma}\:\:vol\left(\mathcal{K}_{2^{-j}\varepsilon}\right) \\
& \lesssim & \sum_{j=0}^\infty \left(2^{-j-1}\varepsilon\right)^{-2\gamma}\:\left(2^{-j}\varepsilon\right)^{3-\alpha} \\
&=& \frac{2^{2\gamma} \varepsilon^{3-\alpha-2\gamma}}{1-2^{-(3-\alpha-2\gamma)}} \\
&<& \infty\:,
\end{eqnarray*}
where we have used the fact that  $2\gamma < 3-\alpha$. Hence $g \in L^2(\Hb^3)$. Consequently, the convolution with the Bessel kernel,
\begin{equation*} 
f:= g * \mathcal{G}_s \in H^s(\Hb^3)\:.
\end{equation*}
We next look at the singularities of $f$. Let $x \in \mathcal{K}_\varepsilon\setminus \mathcal{K}$ and choose $x_0 \in \mathcal{K}$ such that 
\begin{equation*}
d(x,x_0)=d(x,\mathcal{K})=:r\:.
\end{equation*}
We next consider the geodesic ball $B(x_0,r/2)$ and note that for $y \in B(x_0,r/2)$,
\begin{equation*}
d(y,\mathcal{K}) \le d(y,x_0) <\frac{r}{2} < \varepsilon\:.
\end{equation*}
Thus $B(x_0,r/2) \subset \mathcal{K}_\varepsilon$ and hence,
\begin{eqnarray*}
f(x)&=& \int_{\mathcal{K}_\varepsilon} \mathcal{G}_s\left(d(x,y)\right)\:d(y,\mathcal{K})^{-\gamma}\:dvol(y) \\
&\ge & \int_{B(x_0,r/2)} \mathcal{G}_s\left(d(x,y)\right)\:d(y,\mathcal{K})^{-\gamma}\:dvol(y)\:.
\end{eqnarray*}
At this juncture we note that 
\begin{equation*}
d(y,\mathcal{K})^{-\gamma} \ge d(y,x_0)^{-\gamma}\:,
\end{equation*}
and also that for $y \in B(x_0,r/2)$,
\begin{equation*}
d(x,y) \le d(x,x_0)+d(x_0,y) < \frac{3r}{2}\:.
\end{equation*}
Using the above inequalities along with the asymptotic lower bound for the Bessel kernel for small geodesic distances given by Lemma \ref{lemma_bessel_kernel}, we get
\begin{equation*}
f(x) \gtrsim r^{s-3} \int_{B(x_0,r/2)} d(y,x_0)^{-\gamma}\:dvol(y)\:.
\end{equation*} 
Then by the volume growth asymptotics of small geodesic balls given by (\ref{volume_asymptotics}), 
\begin{equation*}
f(x) \gtrsim r^{s-\gamma} \:.
\end{equation*}
Since $r=d(x,\mathcal{K})$, we get that for all $x \in \mathcal{K}_\varepsilon$,
\begin{equation*}
f(x) \gtrsim d(x,\mathcal{K})^{s-\gamma} \:,
\end{equation*}
that is, $f$ is singular on $\mathcal{K}$ (as $\gamma>s$). Then it is easy to see that $f$ is continuous on $\Hb^3 \setminus \overline{\mathcal{K}_\varepsilon}$ (by an easy application of the Dominated Convergence Theorem) and dominated by a continuous function on $\overline{\mathcal{K}_\varepsilon}\setminus \mathcal{K}$. Hence $\mathcal{K}$ is precisely the set of singularities of $f$. This completes the proof of part (1).

For part (2), we first consider the radially symmetric compact set in $\R^3$, considered by Bennett-Rogers in \cite[p.8]{BR} and then take its image under the Riemannian exponential map into $\Hb^3$. Then we proceed as in the proof of part (1). This time, we only additionally use the fact that the Riemannian exponential map is a radial isometry.
\end{proof}

\subsection{An inter-spatial transference principle and the proof of Theorem \ref{H3_theorem}} \label{subsec7.4}
The ideas developed in Subsections \ref{subsec7.1}-\ref{subsec7.3} culminate in the form of the following transference principle:
\begin{theorem} \label{H3_transference}
For any $s>0$ and any phase $\psi$, we have the following transference between the Hausdorff dimensions of (i) divergence sets for $\psi$-dispersive propagations on $\Hb^3$ with radial initial data and (ii) divergence sets for $\wtilde{\psi}$-dispersive propagations on $\R^3$ with radial initial data,
\begin{equation*}
\wtilde{\fd}_{\psi,\Hb^3}(s)=\wtilde{\fd}_{\wtilde{\psi},3}(s)\:,
\end{equation*}
where $$\wtilde{\psi}(\lambda):=\psi\left(\sqrt{\lambda^2+1}\right)\:,\:\:\lambda \ge 0\:.$$
\end{theorem}

\begin{proof}
By Lemma \ref{H3-Sobolev-embedding} and part (2) of Theorem \ref{sobolev_singularity_theorem} (Lemma \ref{Annulus-Sobolev-embedding} and Remark \ref{BR_remark} resp.) the function $\wtilde{\fd}_{\psi,\Hb^3}$ ($\wtilde{\fd}_{\wtilde{\psi},3}$ resp.) is well-defined for all $s>0$, with $\wtilde{\fd}_{\psi,\Hb^3}(s) \ge 3-2s$ ($\wtilde{\fd}_{\wtilde{\psi},3}(s) \ge 3-2s$ resp.) in the range $s \in (0,1/2)$ and $\wtilde{\fd}_{\psi,\Hb^3}(s)=0$ for all $s>1/2$ by Lemma \ref{H3-lemma_s>1/2} ($\wtilde{\fd}_{\wtilde{\psi},3}(s)=0$ for all $s>1/2$ by Lemma \ref{lemma_s>1/2} resp.). 

Let $s_0>0$. Then by the definition, 
\begin{equation*}
dim_{\cH}\D_{\wtilde{\psi}}(f) \le \wtilde{\fd}_{\wtilde{\psi},3}(s_0)\:,\:\:\text{ for all } f \in H^{s_0}(\R^3)_{radial}\:.
\end{equation*}  
So in particular, for any fixed $0<r_1<r_2<\infty$,
\begin{equation*}
dim_{\cH}\left(\D_{\wtilde{\psi}}(f) \cap \mathbb{A}(r_1,r_2)\right) \le \wtilde{\fd}_{\wtilde{\psi},3}(s_0)\:,\:\:\text{ for all } f \in H^{s_0}(\R^3)_{radial}\:.
\end{equation*} 
Thus given $\varepsilon>0$, we have for $\alpha:=\wtilde{\fd}_{\wtilde{\psi},3}(s_0)+\varepsilon$,
\begin{equation*}
\cH^\alpha\left(\D_{\wtilde{\psi}}(f) \cap \mathbb{A}(r_1,r_2)\right)=0\:,\:\:\text{ for all } f \in H^{s_0}(\R^3)_{radial}\:.
\end{equation*}
Then by Frostmann's lemma,
\begin{equation*}
\nu\left(\D_{\wtilde{\psi}}(f) \cap \mathbb{A}(r_1,r_2)\right)=0\:,\:\:\forall \nu \in \M^\alpha(\mathbb{A}(r_1,r_2))\:,\:\:\forall f \in H^{s_0}(\R^3)_{radial}\:.
\end{equation*}
This means that for all $\nu \in \M^\alpha(\mathbb{A}(r_1,r_2))$ and for all $f \in H^{s_0}(\R^3)_{radial}$,
\begin{equation*} 
e^{it\wtilde{\psi}(\sqrt{-\Delta})} f(x) \to f(x)\:,\:\text{ as } t \to 0\:,\: \nu-\text{a.e. } x \in \mathbb{A}(r_1,r_2)\:\:.
\end{equation*}
Then proceeding as in the Step 1 of the proof of Theorem \ref{transference_theorem}, we obtain that there exists a $\nu$-measurable weight $w_\nu$ with $0<w_\nu\le 1,\:\nu$-a.e. on $\mathbb{A}(r_1,r_2)$ such that the following weighted, weak, fractal maximal estimate holds,
\begin{equation*}
\left\|S^*_{\wtilde{\psi}}f\right\|_{L^{2,\infty}(w_\nu\:d\nu)} \lesssim \|f\|_{H^{s_0}(\R^3)}\:,\:\:\text{ for all } f \in \Ss(\R^3)_{radial}\:.
\end{equation*} 
Next proceeding as in the proof of Lemma \ref{H3-N-maximal-estimate-radial}, we get the inter-spatial transference of fractal maximal estimates: for all $\mu := (\exp_o)_* \nu \in \M^\alpha(A(r_1,r_2))$,
\begin{equation*}
\left\|S^*_{\psi}f\right\|_{L^{2,\infty}\left(\left(w_\nu\:\circ\: \exp^{-1}_o\right)\:d\mu \right)} \lesssim \|f\|_{H^{s_0}(\Hb^3)}\:,
\end{equation*}
holds for all $f \in \Ss^2(\Hb^3)_{radial}$. Now by the standard method of maximal functions using the above maximal estimate, density of $\Ss^2(\Hb^3)_{radial}$ and the fractal Sobolev embedding for the initial data (Lemma \ref{H3-Sobolev-embedding}), we get that for all $\mu \in \M^\alpha(A(r_1,r_2))$ and all $f \in H^{s_0}(\Hb^3)_{radial}$,
\begin{equation*}
\mu \left\{x \in A(r_1,r_2) : \displaystyle\lim_{t \to 0} e^{it\psi\left(\sqrt{-\Delta_{\Hb^3}}\right)}f(x) \ne f(x)\right\}=0\:.
\end{equation*}
Thus by Frostmann's lemma,
\begin{equation*}
\mathcal{H}^\alpha  \left\{x \in A(r_1,r_2) : \displaystyle\lim_{t \to 0} e^{it\psi\left(\sqrt{-\Delta_{\Hb^3}}\right)}f(x) \ne f(x)\right\}=0\:,
\end{equation*}
for all $f \in H^{s_0}(\Hb^3)_{radial}$. By countable additivity of Hausdorff measure, we get that
\begin{equation*}
\mathcal{H}^\alpha\left(\D_{\psi}(f)\right)=\mathcal{H}^\alpha \left\{x \in \Hb^3 : \displaystyle\lim_{t \to 0} e^{it\psi\left(\sqrt{-\Delta_{\Hb^3}}\right)}f(x) \ne f(x)\right\}=0\:,
\end{equation*}
for all $f \in H^{s_0}(\Hb^3)_{radial}$. This implies that 
\begin{equation*} 
\wtilde{\fd}_{\psi,\Hb^3}(s_0) \le \alpha = \wtilde{\fd}_{\wtilde{\psi},3}(s_0)+\varepsilon\:.
\end{equation*}
Since, the above inequality is true for any $\varepsilon>0$, we get that
\begin{equation*} 
\wtilde{\fd}_{\psi,\Hb^3}(s_0) \le \wtilde{\fd}_{\wtilde{\psi},3}(s_0)\:.
\end{equation*}
By the symmetry of the argument, switching the roles of $\psi$ and $\wtilde{\psi}$, we get the opposite inequality and thus getting the result.
\end{proof}

Before proceeding to the proof of Theorem \ref{H3_theorem}, we need the following lemma:
\begin{lemma}\cite[Theorem 2]{LLN} \label{chinese_lemma}
For disperive equations in $\R^3$ (\ref{dispersive}) with phases $\wtilde{\psi}$ satisfying,
\begin{enumerate}
\item there exists $m_1>1$ such that
\begin{equation*}
|\wtilde{\psi}'(\lambda)| \asymp \lambda^{m_1-1}\:,\:\: |\wtilde{\psi}''(\lambda)| \gtrsim \lambda^{m_1-2}\:,\:\:\lambda \in (0,1)\:; 
\end{equation*}
\item there exists $m_2>1$ such that
\begin{equation*}
|\wtilde{\psi}'(\lambda)| \asymp \lambda^{m_2-1}\:,\:\: |\wtilde{\psi}''(\lambda)| \gtrsim \lambda^{m_2-2}\:,\:\:\lambda \ge 1\:; 
\end{equation*}
\item either $\wtilde{\psi}''(\lambda)>0$ or $\wtilde{\psi}''(\lambda)<0$, for all $\lambda>0$, 
\end{enumerate}
we have
\begin{equation*}
\wtilde{\fd}_{\tilde{\psi},3}(s) = \begin{cases}
3-2s\:,\:\:&\frac{1}{4} \le s <\frac{1}{2}\\
0\:,\:\:&\frac{1}{2}<s\:.
\end{cases}
\end{equation*}
\end{lemma}

\begin{proof}[Proof of Theorem \ref{H3_theorem}]
By the inter-spatial transference principle Theorem \ref{H3_transference}, the results for $\wtilde{\fd}_{\psi,\Hb^3}$  follow by the corresponding results for  $\wtilde{\fd}_{\tilde{\psi},3}$ where $\tilde{\psi}(\lambda)=\psi\left(\sqrt{\lambda^2+1}\right)$. 

For the PDEs appearing in part (1) of Theorem \ref{H3_theorem}, their corresponding $\tilde{\psi}$ satisfy the conditions of Lemma \ref{chinese_lemma}. In particular, their second derivatives are strictly positive and moreover satisfy the hypothesis (1) and (2) of Lemma \ref{chinese_lemma} as presented in Table \ref{table:2}.

\begingroup
\setlength{\tabcolsep}{10pt} 
\renewcommand{\arraystretch}{1.5} 
\begin{table}[H]
\centering
\begin{tabular}{ ||p{5cm}|| p{4cm} || p{1cm} || p{1cm}|| }
 \hline
 Dispersive equation & Corresponding  $\wtilde{\psi}(\lambda)$  & $m_1$ & $m_2$ \\
 \hline 
  Schr\"odinger &  $\lambda^2+1$   & 2 & 2    \\
 \hline
Fractional Schr\"odinger with convex phase & $\left(\lambda^2+1\right)^{\frac{a}{2}}\:,\:\:a>1$  & 2 & $a$\\
  \hline
Beam & $\left(1+\left(\lambda^2+1\right)^2\right)^{\frac{1}{2}}$ & 2 & 2   \\
\hline
Linearized Boussinesq & $\left(\lambda^2+1\right)^{\frac{1}{2}} \left(\lambda^2+2\right)^{\frac{1}{2}}$ & 2& 2 \\
\hline
\end{tabular}

\medskip

\caption{Dispersive PDEs in part (1)}
\label{table:2}
\end{table}
\endgroup

Then by Lemma \ref{chinese_lemma}, the results follow for the range $s \in [1/4,1/2)\cup (1/2,\infty)$. The results for $s\in (0,1/4)$, on the other hand, follow from \cite[Theorem 1.1]{Dewan2}. This completes the proof of Part (1).

Part (2) ((3) resp.) of Theorem \ref{H3_theorem} follows from the inter-spatial transference principle Theorem \ref{H3_transference} and Corollary \ref{radial_corollary2} (Corollary \ref{radial_corollary3} resp.), as for any $C>0$ and $a \in (0,1]$,
\begin{equation*}
\left(\lambda^2+C\right)^{\frac{a}{2}}=\lambda^a\:+\:\mathcal{O}(1)\:,\:\:\lambda \gg 1\:.
\end{equation*}
This completes the proof of Theorem \ref{H3_theorem}.
\end{proof}

\section{Concluding remarks}
\label{sec8}
In this section, we make some remarks and pose some new problems. 
\begin{enumerate}
\item For the fractional Schr\"odinger equations on $\R^n$ with convex or concave phases i.e. $a>1$ or $a \in (0,1)$ and radial initial data (see (\ref{BR_estimates}) and Theorem \ref{radial_concave_thm}), the quantity $\tilde{\fd}_{\psi,n}(1/2)$ continues to remain unknown and determining it seems to be an interesting problem.

\medskip

\item It is a natural and interesting problem to generalize Theorem \ref{H3_theorem} to real hyperbolic spaces of all dimensions and moreover, to the class of homogeneous manifolds known as rank one Riemannian symmetric spaces of non-compact type.
\end{enumerate}

\section*{Declarations}
\subsection*{Acknowledgements}  The author is supported by the Institute Post Doctoral Fellowship of the Indian Institute of Technology, Bombay.

\subsection*{Data availability:} No data was used for the research described in the article.
\subsection*{Competing interests:} The author declares no competing interests.
\subsection*{AI declaration:} No artificial intelligence tools were used. 

\bibliographystyle{amsplain}

\end{document}